\documentclass[11pt]{article}
\usepackage{epsf}
\usepackage{amsmath}
\usepackage{epsfig}
\usepackage{times}
\usepackage{amssymb}
\usepackage{amsthm}
\usepackage{setspace}
\usepackage{cite}

\usepackage{algorithmic}  
\usepackage{algorithm}

\usepackage{shadow}
\usepackage{fancybox}
\usepackage{fancyhdr}

\def\y{{\bf y}}

\def\x{{\bf x}}

\def\x{{\mathbf x}}

\def\x{{\bf x}}
\def\y{{\bf y}}

\def\b{{\bf b}}

\def\h{{\bf h}}

\def\be{\begin{equation}}
\def\ee{\end{equation}}
\def\ba{\left[\begin{array}}
\def\ea{\end{array}\right]}

\def\t{{\bf t}}

\def\x{{\bf x}}
\def\y{{\bf y}}

\def\b{{\bf b}}

\def\1{{\bf 1}}

\def\g{{\bf g}}
\def\0{{\bf 0}}

\def\Sric{S_{ric}}

\newtheorem{theorem}{Theorem}
\newtheorem{corollary}{Corollary}

\newtheorem{lemma}{Lemma}

\begin{document}

\begin{singlespace}

\title {Bounds on restricted isometry constants of random matrices 
\footnote{ This work was supported in part by NSF grant \#CCF-1217857.}
}
\author{
\textsc{Mihailo Stojnic}
\\
\\
{School of Industrial Engineering}\\
{Purdue University, West Lafayette, IN 47907} \\
{e-mail: {\tt mstojnic@purdue.edu}} }
\date{}
\maketitle

\centerline{{\bf Abstract}} \vspace*{0.1in}

In this paper we look at isometry properties of random matrices. During the last decade these properties gained a lot attention in a field called compressed sensing in first place due to their initial use in \cite{CRT,CT}. Namely, in \cite{CRT,CT} these quantities were used as a critical tool in providing a rigorous analysis of $\ell_1$ optimization's ability to solve an under-determined system of linear equations with sparse solutions. In such a framework a particular type of isometry, called restricted isometry, plays a key role. One then typically introduces a couple of quantities, called upper and lower restricted isometry constants to characterize the isometry properties of random matrices. Those constants are then usually viewed as mathematical objects of interest and their a precise characterization is desirable. The first estimates of these quantities within compressed sensing were given in \cite{CRT,CT}. As the need for precisely estimating them grew further a finer improvements of these initial estimates were obtained in e.g. \cite{BCTsharp09,BT10}. These are typically obtained through a combination of union-bounding strategy and powerful tail estimates of extreme eigenvalues of Wishart (Gaussian) matrices (see, e.g. \cite{Edelman88}). In this paper we attempt to circumvent such an approach and provide an alternative way to obtain similar estimates.

\vspace*{0.25in} \noindent {\bf Index Terms: Restricted isometry constants; compressed sensing; $\ell_1$-minimization}.

\end{singlespace}

\section{Introduction}
\label{sec:back}

In this paper we look at isometry properties of random matrices. Our motivation comes from their initial employment for the analysis of $\ell_1$-optimization success in solving under-determined linear systems with sparse solutions. In \cite{CRT,CT} the following classic inverse linear problem was considered: consider a $m\times n$ system matrix $A$ with real components. Let $\tilde{\x}$ be a vector with no more than $k$ nonzero components (we will call such a vector $k$-sparse). Further let
\begin{equation}
\y=A\tilde{\x}.\label{eq:defy}
\end{equation}
Then one can pose the inverse problem: given $\y$ and $A$ can one then recover $\tilde{\x}$? The answer critically depends on the structure of $A$ and relations between $k$, $m$, and $n$. To avoid any special case we will assume that $A$ is always a full rank matrix and that $k<m<n$. Moreover, to simplify the exposition we will assume that $n$ is large and the so-called linear regime, i.e. we will assume that $k=\beta n$ and $m=\alpha n$ where $\beta$ and $\alpha$ are constants independent of $n$. It is then a relatively easy algebraic exercise to show that if $\beta<\alpha/2$ the solution to the above problem is unique and equal to $\tilde{\x}$. On the other hand if $\beta>\alpha/2$, roughly speaking, the ``odds" are pretty good that the solution is unique and equal to $\tilde{\x}$. Equipped with these algebraic facts one then faces the problem of actually figuring out what $\tilde{\x}$ really is, if $\y$ and $A$ from (\ref{eq:defy}) are given. That essentially (loosely speaking) boils down to finding the sparsest solution of the following under-determined system of linear equations
\begin{equation}
A\x=\y. \label{eq:system}
\end{equation}
The above problem is of course hard. Moreover it is a mathematical cornerstone of the field called compressed sensing that has seen an unprecedented expansion in recent years (way more about the compressed sensing conception and various problems of interest within the fields that grew out of the above mentioned  basic compressed sensing concept can be found in a tone of references; here we point out to a couple of introductory papers, e.g. \cite{DOnoho06CS,CRT}).

Looking back at (\ref{eq:system}), clearly one can consider an exhaustive search type of solution where one would look at all subsets of $k$ columns of $A$ and then attempt to solve the resulting system. However, in the linear regime that we assumed above such an approach becomes prohibitively slow as $n$ grows. That of course led in last several decades towards a search for more clever algorithms for solving (\ref{eq:system}). Many great algorithms were developed (especially during the last decade) and many of them have even provably excellent performance measures (see, e.g. \cite{JATGomp,JAT,NeVe07,DTDSomp,NT08,DaiMil08,DonMalMon09}). A particularly successful strategy is the following so-called $\ell_1$-optimization technique (variations of the standard $\ell_1$-optimization from e.g.
\cite{CWBreweighted,SChretien08,SaZh08}) as well as those from \cite{SCY08,FL08,GN03,GN04,GN07,DG08} related to $\ell_q$-optimization, $0<q<1$
are possible as well)
\begin{eqnarray}
\mbox{min} & & \|\x\|_{1}\nonumber \\
\mbox{subject to} & & A\x=\y. \label{eq:l1}
\end{eqnarray}
It has been known for a long time that the solution to the above problem is fairly often $\tilde{\x}$ in (\ref{eq:system}). It is however the work of \cite{CRT,CT,DOnoho06CS} that for the first time established it as a rigorous mathematical fact in a certain statistical scenario for the linear regime that we consider here  (more
on the non-linear regime, i.e. on the regime when $m$ is larger than
linearly proportional to $k$ can be found in e.g.
\cite{CoMu05,GiStTrVe06,GiStTrVe07}). On the path to establishing this fact \cite{CRT,CT} made a use of isometry properties of matrix $A$. Namely, they observed that if one looks at $k$- column subsets of $A$ and can somehow show that they typically behave as isometries one can then guarantee that the solution of (\ref{eq:l1}) is $\tilde{\x}$. To make the above description of such an observation more precise it is more convenient to define the following objects (for definitions of related, similar objects see, e.g. \cite{CRT,CT, Crip,BT10}):
\begin{eqnarray}
\xi_{uric}(\beta,\alpha) & = & \max_{\|\x\|_{2}=1,\|\x\|_{\ell_0}=k}\|A\x\|_2\nonumber \\
\xi_{lric}(\beta,\alpha) & = & \min_{\|\x\|_{2}=1,\|\x\|_{\ell_0}=k}\|A\x\|_2, \label{eq:ulric}
\end{eqnarray}
where $\|\x\|_{\ell_0}$ is the so-called $\ell_0$-norm which for all practical purposes counts how many nonzero components vector $\x$ has. Now, if one assumes that the columns of $A$ are normalized so that they all have unit Euclidean norm then how far away from $1$ are $\xi_{uric}(\beta,\alpha)$ and $\xi_{lric}(\beta,\alpha)$ is what determines how close $A$ is to satisfying restricted isometry properties. What was observed in \cite{CRT,CT} is essentially what kind of effect will deviation of $\xi_{uric}(\beta,\alpha)$ and $\xi_{lric}(\beta,\alpha)$ from $1$ have on the ability of (\ref{eq:l1}) to recover $\tilde{\x}$ from $\ref{eq:defy}$. All these things were of course rigorously quantified as well assuming a statistical scenario. In such a scenario matrix $A$ is often assumed to have appropriately scaled i.i.d. standard normal components. We will make a similar assumption throughout the rest of the paper as well (however, we do mention that our results are in no way restricted only to such matrices $A$; in fact we will briefly towards the end of the paper discuss the generality of the presented results as well). Namely, to ease the exposition we will assume that the elements of $A$ are i.i.d. standard normal components. Our goal will be to provide estimates for $\xi_{uric}(\beta,\alpha)$ and $\xi_{lric}(\beta,\alpha)$ in such a statistical scenario.

We should also mention that the restricted isometry properties that were considered in \cite{CRT,CT} are not the only way how one can analyze the ability of (\ref{eq:l1}) to recover $\tilde{\x}$ in (\ref{eq:system}). Namely, in \cite{DonohoPol,DonohoUnsigned}, an alternative approach based on high-dimensional ``random" geometry was presented. Moreover, such an approach was capable of providing the exact relations between $k$, $m$, and $n$ (essentially ($\beta,\alpha$) relations) so that (\ref{eq:l1}) typically in a statistical scenario recovers $\tilde{\x}$. In our own series of work \cite{StojnicICASSP09,StojnicCSetam09,StojnicUpper10,StojnicEquiv10}, we designed an alternative probabilistic approach that was also able to provide the exact ($\beta,\alpha$) relations so that (\ref{eq:l1}) typically in a statistical scenario recovers $\tilde{\x}$. However, for the purposes of this paper we believe that the analysis presented in \cite{CRT,CT} and later in \cite{BCTT09} is more relevant.

Of course before proceeding with the presentation of our main results, we should mention that after the original considerations in \cite{CRT,CT}, the restricted isometry properties have found a great deal of applications in various other studies related to linear inverse problems as well as in studies that viewed them as pure mathematical objects (see, e.g. \cite{Crip,CRT,Bar,Ver,ALPTJ09}). Along the same lines, we should mention that our motivation and interest come from the initial types of analysis used to study $\ell_1$-optimization properties. However, our presentation and contribution view them as purely mathematical objects and all results we present are a purely mathematical characterization of restricted isometry properties of random matrices $A$ (which essentially boils down to an as precise as possible estimate of $\xi_{uric}(\beta,\alpha)$ and $\xi_{lric}(\beta,\alpha)$ in (\ref{eq:ulric})). Of course there has been a great deal of work in recent years that provided solid estimates for $\xi_{uric}(\beta,\alpha)$ and $\xi_{lric}(\beta,\alpha)$. We should first mention that already in the introductory papers \cite{CRT,CT} pretty good estimates for $\xi_{uric}(\beta,\alpha)$ and $\xi_{lric}(\beta,\alpha)$ were provided. In those papers of course the primary goal was the analysis of (\ref{eq:l1}) and the estimates provided for $\xi_{uric}(\beta,\alpha)$ and $\xi_{lric}(\beta,\alpha)$ were more of an instructional nature. In \cite{BCTsharp09} and \cite{BT10} the strategy from \cite{CRT,CT} (based on a combination of union-bounding and fairly precise tail estimates of extreme eigenvalues of Wishart matrices) was refined and better (closer to $1$) values for $\xi_{uric}(\beta,\alpha)$ and $\xi_{lric}(\beta,\alpha)$ were obtained. We will throughout the paper recall on some of these results and will discuss them in more detail as we present our own. At this point, we would like to emphasize that the results that we will present will provide a fairly good set of estimates for both $\xi_{uric}(\beta,\alpha)$ and $\xi_{lric}(\beta,\alpha)$. However, rather then particular values, it is the mechanisms that we designed to obtained them that we believe are of particular value. Essentially, the framework that we designed attempts to circumvent the traditional union-boudning/Wishart extreme eigenvalues approach.

Before proceeding further we briefly mention how the rest of the paper is organized. In Section \ref{sec:xiu} we present a mechanism that can be used to provide an upper bound on $\xi_{uric}$ (from this point on we will fairly often instead of $\xi_{uric}(\beta,\alpha)$ and $\xi_{lric}(\beta,\alpha)$ write just $\xi_{uric}$ and $\xi_{lric}$, respectively). In Section \ref{sec:xiulow}, we provide a way to improve the results presented in Section \ref{sec:xiu} (this will rely on a substantial progress we recently made in studying various other combinatorial problems in e.g. \cite{StojnicLiftStrSec13,StojnicMoreSophHopBnds10}). In Section \ref{sec:xil} we then present a counterpart to the mechanism from Section \ref{sec:xiu} that can be used to provide a lower bound on $\xi_{lric}$. Along the same lines, we then in Section \ref{sec:xillift} provide a counterpart to the mechanism from Section \ref{sec:xiulow} that can be used to lift the lower bounds on $\xi_{lric}$. Finally in Section \ref{sec:conc} we present a brief discussion and provide a few concluding remarks related to the obtained results.

\section{Bounding $\xi_{uric}$}
\label{sec:xiu}

In this section we look at $\xi_{uric}$ and design a mechanism that can be used to upper-bound it. The mechanism will to an extent be related to the mechanism we presented in \cite{StojnicCSetam09} and used for the analysis of (\ref{eq:l1})'s ability to recover $\tilde{\x}$. Throughout the presentation in this and all subsequent sections we will consequently assume a substantial level of familiarity with many of the well-known results that relate to the performance characterization of (\ref{eq:l1}) (we will fairly often recall on many results/definitions that we established in \cite{StojnicCSetam09,StojnicLiftStrSec13}). We start by defining a set $\Sric$
\begin{equation}
\Sric=\{\x\in S^{n-1}| \quad \|\x\|_{\ell_0}=k\},\label{eq:defSric}
\end{equation}
where $S^{n-1}$ is the unit sphere in $R^n$. Then one can transform the first part of (\ref{eq:ulric}) in the following way
\begin{equation}
\xi_{uric}=\max_{\x\in\Sric}\|A\x\|_2.\label{eq:negham1}
\end{equation}
A very similar set of problems was considered in \cite{StojnicCSetam09,StojnicHopBnds10}. A powerful set of upper/lower bounds was established in \cite{StojnicCSetam09,StojnicHopBnds10} on various problems considered there. Here, using mechanism similar to those from \cite{StojnicCSetam09,StojnicHopBnds10} we will establish a similar set of upper bounds on $\xi_{uric}$. However, one should note that the structure of set $\Sric$ is somewhat different than the structure of sets considered in \cite{StojnicCSetam09,StojnicHopBnds10} and a careful approach will be needed to readapt the mechanisms from \cite{StojnicCSetam09,StojnicHopBnds10} to the problem we consider here. Also, the mechanisms of \cite{StojnicCSetam09,StojnicHopBnds10} were powerful enough to establish the concentration of quantities similar to $\xi_{uric}$. Moreover, these quantities concentrate around their mean values. It will therefore be enough for us to only view $E\xi_{uric}$. Below we present a way to create an upper-bound on the optimal value of $E\xi_{uric}$.

\subsection{Probabilistic approach to upper bounding $\xi_{uric}$}
\label{sec:xiuprob}

In this section we look at $E\xi_{uric}$ and design its an upper-bound. To do so we rely on the following lemma (which is a modified version of a similar lemma from \cite{StojnicHopBnds10} and, as mentioned in \cite{StojnicHopBnds10}, a direct application of Theorem $4$ from \cite{StojnicHopBnds10} proven in various forms and shapes in e.g. \cite{Gordon85,Slep62}):
\begin{lemma}
Let $A$ be an $m\times n$ matrix with i.i.d. standard normal components. Let $\g$ and $\h$ be $n\times 1$ and $m\times 1$ vectors, respectively, with i.i.d. standard normal components. Also, let $g$ be a standard normal random variable. Then
\begin{equation}
E(\max_{\x\in\Sric,\|\y\|_2=1}(\y^T A\x +\|\x\|_2 g))\leq E(\max_{\x\in\Sric,\|\y\|_2=1}(\|\x\|_2\g^T\y+\h^T\x)).\label{eq:posexplemma}
\end{equation}\label{lemma:posexplemma}
\end{lemma}
\begin{proof}
As mentioned above, the proof is a standard/direct application of Theorem $4$ from \cite{StojnicHopBnds10}. We skip the details and mention that the only difference between the proof one needs here and the one given in \cite{StojnicHopBnds10} is the structure of set $\Sric$. However, such a difference changes nothing in the remainder of the proof.
\end{proof}

Using results of Lemma \ref{lemma:posexplemma} we then have
\begin{multline}
E(\max_{\x\in\Sric} \|A\x\|_2) =E(\max_{\x\Sric,\|\y\|_2=1}(\y^T A\x +\|\x\|_2 g))\\\leq E(\max_{\x\in\Sric,\|\y\|_2=1}(\|\x\|_2\g^T\y+\h^T\x))=E\|\x\|_2\|\g\|_2+E\max_{\x\in\Sric}\h^T\x\leq \sqrt{m}+E\max_{\x\in\Sric}\h^T\x.\label{eq:poshopaftlemma2}
\end{multline}
Let $\bar{\h}$ be the vector of magnitudes of $\h$ sorted in nondecreasing order (of course, ties are broken arbitrarily). Then from (\ref{eq:poshopaftlemma2}) we have
\begin{equation}
E(\max_{\x\in\Sric} \|A\x\|_2)\leq \sqrt{m}+E\sqrt{\sum_{i=n-k+1}^{n}\bar{\h}_i}\leq \sqrt{m}+\sqrt{E\sum_{i=n-k+1}^{n}\bar{\h}_i}.\label{eq:poshopaftlemma3}
\end{equation}
Using the results of \cite{StojnicCSetam09} one then has
\begin{equation}
\lim_{n\rightarrow \infty}\frac{E(\max_{\x\in\Sric} \|A\x\|_2)}{\sqrt{m}}\leq 1+\sqrt{\lim_{n\rightarrow \infty}\frac{E\sum_{i=n-k+1}^{n}\bar{\h}_i^2}{\alpha n}}= 1+\frac{1}{\sqrt{\alpha}}\sqrt{\beta+\frac{2\mbox{erfinv}(1-\beta)}{\sqrt{\pi}e^{(\mbox{erfinv}(1-\beta))^2}}}.\label{eq:poshopaftlemma4}
\end{equation}
Connecting beginning and end of (\ref{eq:poshopaftlemma4}) we finally have an upper bound on $E\xi_{uric}$ (in a scaled more appropriate form),
\begin{equation}
\lim_{n\rightarrow \infty}\frac{E\xi_{uric}}{\sqrt{m}}=\frac{E(\max_{\x\in\Sric} \|A\x\|_2)}{\sqrt{m}} \leq 1+\frac{1}{\sqrt{\alpha}}\sqrt{\beta+\frac{2\mbox{erfinv}(1-\beta)}{\sqrt{\pi}e^{(\mbox{erfinv}(1-\beta))^2}}}.\label{eq:poshopubexp1}
\end{equation}

We summarize our results from this subsection in the following lemma.

\begin{lemma}
Let $A$ be an $m\times n$ matrix with i.i.d. standard normal components. Let $n$ be large and let $k=\beta n$, $m=\alpha n$, where $\beta,\alpha>0$ are constants independent of $n$. Let $\xi_{uric}$ be as in (\ref{eq:negham1}).
\begin{equation}
\lim_{n\rightarrow \infty}\frac{E\xi_{uric}}{\sqrt{m}}=\frac{E(\max_{\x\in\Sric} \|A\x\|_2)}{\sqrt{m}} \leq 1+\frac{1}{\sqrt{\alpha}}\sqrt{\beta+\frac{2\mbox{erfinv}(1-\beta)}{\sqrt{\pi}e^{(\mbox{erfinv}(1-\beta))^2}}}.\label{eq:uricexplemma}
\end{equation}
Moreover, let $\xi_{uric}^{(u)}$ be a quantity such that
\begin{equation}
1+\frac{1}{\sqrt{\alpha}}\sqrt{\beta+\frac{2\mbox{erfinv}(1-\beta)}{\sqrt{\pi}e^{(\mbox{erfinv}(1-\beta))^2}}}<\xi_{uric}^{(u)}.
\label{eq:conduriclemma}
\end{equation}
Then
\begin{eqnarray}
& & \lim_{n\rightarrow\infty}P(\max_{\x\in\Sric}(\|A\x\|_2)\leq \xi_{uric}^{(u)}\sqrt{m})\geq 1\nonumber \\
& \Leftrightarrow & \lim_{n\rightarrow\infty}P(\xi_{uric}\leq \xi_{uric}^{(u)}\sqrt{m})\geq 1 \nonumber \\
& \Leftrightarrow & \lim_{n\rightarrow\infty}P(\xi_{uric}^2\leq (\xi_{uric}^{(u)})^2 m)\geq 1. \label{eq:uricproblemma}
\end{eqnarray}
\label{lemma:uriclemma}
\end{lemma}
\begin{proof}
The proof of (\ref{eq:uricexplemma}) follows from (\ref{eq:poshopubexp1}) and the above discussion. The proof of the moreover part follows from the concentration properties considered in \cite{StojnicCSetam09} and the corresponding discussion presented in \cite{StojnicHopBnds10}.
\end{proof}

\subsection{Numerical results -- upper bound on $\xi_{uric}$}
\label{sec:xiuprobnum}

In this subsection we present a small collection of numerical results one can obtain based on Lemma \ref{lemma:uriclemma}. In Tables \ref{tab:urictab1} and \ref{tab:urictab2} we essentially show the upper bounds on $\lim_{n\rightarrow\infty}\frac{E\xi_{uric}}{\sqrt{m}}$ one can obtain based on the above lemma. We refer to those bounds as $\xi_{uric}^{(u)}$. Also, to get a feeling how far off they could be from the optimal ones we also show a set of known bounds from \cite{BT10} (based on numerical experiments conducted in \cite{BT10} those appeared as if not that far away from the optimal values). While there are other ways that can be used to compute bounds on
$\lim_{n\rightarrow\infty}\frac{E\xi_{uric}}{\sqrt{m}}$, we chose to present the results obtained through the concepts developed in \cite{BT10} for two reasons: 1) the calculations behind these bounds are fairly simple and 2) the main idea behind their construction is very neat (alternatively one can also look at the results from e.g. \cite{CRT,CT,BCTsharp09}; the results from \cite{BT10} however provide lower values of the upper bounds; for a detailed discussion how the results from \cite{CRT,CT,BCTsharp09,BT10} relate to each other we refer to \cite{BT10}). We denote the upper bounds on $\lim_{n\rightarrow\infty}\frac{E\xi_{uric}}{\sqrt{m}}$ that one can obtain based on \cite{BT10} as $\xi_{uric}^{BT}$. Also, we do mention that the values presented in Tables \ref{tab:urictab1} and \ref{tab:urictab2} are slightly modified versions of the corresponding quantities from \cite{BT10}. Namely, to get a complete agreement with \cite{BT10} one should think of ${\cal U}$ in \cite{BT10} as $(\xi_{uric}^{BT})^2-1$ (or in other words, what we call $\xi_{uric}^{BT}$ in \cite{BT10} is called $\lambda^{max}$). Overall, the results obtained based on Lemma \ref{lemma:uriclemma} improve a bit on those from \cite{BT10} and the improvement becomes more visible as ratio $\beta/\alpha$ grows.

\begin{table}
\caption{Upper bounds on $\lim_{n\rightarrow\infty}\frac{E\xi_{uric}}{\sqrt{m}}$ -- low $\beta/\alpha\leq 0.5$ regime}\vspace{.1in}
\hspace{-0in}\centering
\begin{tabular}{||c|c|c|c|c|c||}\hline\hline
 $\alpha$       & $0.1$ & $0.3$ & $0.5$ & $0.7$ & $0.9$  \\ \hline\hline
 $\beta/\alpha=0.1$; $\xi_{uric}^{BT}$       & $1.9786$ & $1.8970$ & $1.8562$ & $1.8280$ & $1.8062$ \\ \hline
 $\beta/\alpha=0.1$; $\xi_{uric}^{(u)}$      & $1.9192$ & $1.8049$ & $1.7471$ & $1.7071$ & $1.6761$ \\ \hline\hline
 $\beta/\alpha=0.3$; $\xi_{uric}^{BT}$       & $2.5822$ & $2.4067$ & $2.3142$ & $2.2471$ & $2.1925$ \\ \hline
 $\beta/\alpha=0.3$; $\xi_{uric}^{(u)}$      & $2.3941$ & $2.1710$ & $2.0560$ & $1.9753$ & $1.9123$ \\ \hline\hline
 $\beta/\alpha=0.5$; $\xi_{uric}^{BT}$       & $2.9622$ & $2.7036$ & $2.5591$ & $2.4479$ & $2.3508$ \\ \hline
 $\beta/\alpha=0.5$; $\xi_{uric}^{(u)}$      & $2.6706$ & $2.3633$ & $2.2030$ & $2.0901$ & $2.0017$ \\ \hline\hline
\end{tabular}
\label{tab:urictab1}
\end{table}

\begin{table}
\caption{Upper bounds on $\lim_{n\rightarrow\infty}\frac{E\xi_{uric}}{\sqrt{m}}$ -- high $\beta/\alpha> 0.5$ regime}\vspace{.1in}
\hspace{-0in}\centering
\begin{tabular}{||c|c|c|c|c|c||}\hline\hline
 $\alpha$       & $0.1$ & $0.3$ & $0.5$ & $0.7$ & $0.9$  \\ \hline\hline
 $\beta/\alpha=0.7$; $\xi_{uric}^{BT}$        & $3.2505$ & $2.9094$ & $2.7053$ & $2.5337$ & $2.3769$ \\ \hline
 $\beta/\alpha=0.7$; $\xi_{uric}^{(u)}$       & $2.8709$ & $2.4898$ & $2.2898$ & $2.1489$ & $2.0394$ \\ \hline\hline
 $\beta/\alpha=0.9$; $\xi_{uric}^{BT}$        & $3.4849$ & $3.0577$ & $2.7779$ & $2.5385$ & $2.3769$ \\ \hline
 $\beta/\alpha=0.9$; $\xi_{uric}^{(u)}$       & $3.0283$ & $2.5801$ & $2.3440$ & $2.1785$ & $2.0522$ \\ \hline\hline
\end{tabular}
\label{tab:urictab2}
\end{table}

\section{Lowering $\xi_{uric}$'s bounds}
\label{sec:xiulow}

In the previous section we presented a fairly powerful method for estimating $\xi_{uric}$. However, the results we obtained are not exact. Of course, the main reason is an inability to determine the exact value of $E\xi_{uric}$. Instead we resorted to its upper bounds and those could be loose. In this section we will use some of
the ideas we recently introduced in \cite{StojnicLiftStrSec13,StojnicMoreSophHopBnds10} to provide a substantial conceptual improvement in these bounds
which would in turn reflect even in practically better estimates for $E\xi_{uric}$ (as we will see later on, similar concepts will be employed to deal with $E\xi_{lric}$ and practical improvement in those cases will be even more substantial). Below we recall on the main components of the mechanisms introduced in \cite{StojnicMoreSophHopBnds10,StojnicLiftStrSec13} and how these can be adapted to be of use when dealing with problems of interest here.

\subsection{Probabilistic approach to lowering $\xi_{uric}$'s bounds}
\label{sec:xiulowprob}

We start by introducing a lemma very similar to the one considered in \cite{StojnicMoreSophHopBnds10} (the following lemma is essentially a direct consequence/application of Theorem $1$ from \cite{StojnicMoreSophHopBnds10} which of course was proved in \cite{Gordon85} and in a slightly different form earlier in \cite{Slep62}).
\begin{lemma}
Let $A$ be an $m\times n$ matrix with i.i.d. standard normal components. Let $\g$ and $\h$ be $n\times 1$ and $m\times 1$ vectors, respectively, with i.i.d. standard normal components. Also, let $g$ be a standard normal random variable and let $c_3$ be a positive constant. Then
\begin{equation}
E(\max_{\x\in\Sric,\|\y\|_2=1}e^{c_3(\y^T A\x + g)})\leq E(\max_{\x\in\Sric,\|\y\|_2=1}e^{c_3(\g^T\y+\h^T\x)}).\label{eq:posexplemmalift}
\end{equation}\label{lemma:posexplemmalift}
\end{lemma}
\begin{proof}
As mentioned above, the proof is a standard/direct application of Theorem $1$ from \cite{StojnicMoreSophHopBnds10} which was proved in \cite{Gordon85} and in a slightly different form earlier in \cite{Slep62}. The only difference is the structure of $\Sric$ which changes nothing in the proof.
\end{proof}

Following what was done in \cite{StojnicMoreSophHopBnds10} one then has
\begin{equation}
E(\max_{\x\in\Sric}\|A\x\|_2)\leq
-\frac{c_3}{2}+\frac{1}{c_3}\log(E(\max_{\x\in\Sric}(e^{c_3\h^T\x})))
+\frac{1}{c_3}\log(E(\max_{\|\y\|_2=1}(e^{c_3\g^T\y}))).\label{eq:chpos8uriclift}
\end{equation}
Let $c_3=c_3^{(s)}\sqrt{n}$ where $c_3^{(s)}$ is a constant independent of $n$. Then following further what we did in \cite{StojnicMoreSophHopBnds10} we have
\begin{equation}
\frac{E(\max_{\x\in\Sric}\|A\x\|_2)}{\sqrt{n}}
 \leq
-\frac{c_3^{(s)}}{2}+\frac{1}{nc_3^{(s)}}\log(E(\max_{\x\in\Sric}(e^{c_3^{(s)}\sqrt{n}\h^T\x})))
+\frac{1}{nc_3^{(s)}}\log(E(\max_{\|\y\|_2=1}(e^{c_3^{(s)}\sqrt{n}\g^T\y}))),
\end{equation}
or written slightly differently
\begin{eqnarray}
\sqrt{\alpha}\frac{E(\max_{\x\in\Sric}\|A\x\|_2)}{\sqrt{m}}
& \leq &
-\frac{c_3^{(s)}}{2}+\frac{1}{nc_3^{(s)}}\log(E(\max_{\x\in\Sric}(e^{c_3^{(s)}\sqrt{n}\h^T\x})))
+\frac{1}{nc_3^{(s)}}\log(E(\max_{\|\y\|_2=1}(e^{c_3^{(s)}\sqrt{n}\g^T\y}))) \nonumber \\
& = & -\frac{c_3^{(s)}}{2}+I_{uric}(c_3^{(s)},\beta)
+I_{sph}(c_3^{(s)},\alpha),\label{eq:chpos9uriclift}
\end{eqnarray}
where
\begin{eqnarray}
I_{uric}(c_3^{(s)},\beta) & = & \frac{1}{nc_3^{(s)}}\log(E(\max_{\x\in\Sric}(e^{c_3^{(s)}\sqrt{n}\h^T\x})))\nonumber \\
I_{sph}(c_3^{(s)},\alpha) & = & \frac{1}{nc_3^{(s)}}\log(E(\min_{\|\y\|_2=1}(e^{c_3^{(s)}\sqrt{n}\g^T\y}))).\label{eq:defIs}
\end{eqnarray}
In \cite{StojnicMoreSophHopBnds10} we also established the following
\begin{equation}
I_{sph}(c_3^{(s)},\alpha)=\frac{1}{nc_3^{(s)}}\log(Ee^{c_3^{(s)}\sqrt{n}\|\g\|_2})
\doteq\widehat{\gamma^{(s)}}-\frac{\alpha}{2c_3^{(s)}}\log(1-\frac{c_3^{(s)}}{2\widehat{\gamma^{(s)}}})),\label{eq:gamaiden2}
\end{equation}
where (following \cite{SPH}) $\doteq$ stands for equality that holds as $n\rightarrow \infty$ and
\begin{equation}
\widehat{\gamma^{(s)}}=\frac{2c_3^{(s)}+\sqrt{4(c_3^{(s)})^2+16\alpha}}{8}.\label{eq:gamaiden3}
\end{equation}
We also mention that (as in \cite{StojnicMoreSophHopBnds10}) $\doteq$ can be replaced with a trivial inequality $\leq$ for our needs here.

To make the bound in (\ref{eq:chpos9uriclift}) operational, the only thing left to consider is $I_{uric}(c_3^{(s)},\beta)$. We will now naturally switch to consideration of $I_{uric}(c_3^{(s)},\beta)$. However to make the presentation easier to follow first we slightly modify set $\Sric$ in the following way:
\begin{equation}
\Sric=\{\x\in S^{n-1}| \quad \x_i=\b_i\x_i^{\prime},\sum_{i=1}^n \b_i=k,\b_i\in\{0,1\},\b_i=0\Rightarrow\x_i^{\prime}=0,\|\x^{\prime}\|_2=1\},\label{eq:defSriclift}
\end{equation}
where $S^{n-1}$ is the unit sphere in $R^n$. Let $f(\x)=\h^T\x$ and
we start with the following line of identities
\begin{eqnarray}
f_{uric}=\max_{\x\in\Sric}f(\x)=-\min_{\x\in\Sric}-\h^T\x =  -\min_{\b,\x} & & -\h^T\x\nonumber \\
\mbox{subject to} & & \x_i=\b_i\x_i^{\prime},1\leq i\leq n\nonumber \\
& & \|\x_i^{\prime}\|_2^2=1,\nonumber \\
& & \b_i\in\{0,1\},1\leq i\leq n,\nonumber \\
& & \sum_{i=1}^{n}\b_i= k, \nonumber \\
& & \b_i=0\Rightarrow \x_i^{\prime}=0,1\leq i\leq n.\label{eq:uriceq0}
\end{eqnarray}
Let $\phi_i=(\b_i=0\Rightarrow \x_i^{\prime}=0),1\leq i\leq n$. We then further have
\begin{eqnarray}
f_{uric} & = & -\min_{\b_i\in\{0,1\},\phi_i,\x^{\prime}}\max_{\gamma_{uric},\nu_{uric}\geq 0} -\sum_{i=1}^{n}\h_i\b_i\x_i^{\prime}
+\nu_{uric}\sum_{i=1}^{n}\b_i-\nu_{uric}k+\gamma_{uric}\sum_{i=1}^{n} (\x_i^{\prime})^2-\gamma_{uric}\nonumber \\
&\leq & -\max_{\gamma_{uric},\nu_{uric}\geq 0}\min_{\b_i\in\{0,1\},\phi_i,\x^{\prime}} -\sum_{i=1}^{n}\h_i\b_i\x_i^{\prime}
+\nu_{uric}\sum_{i=1}^{n}\b_i-\nu_{uric}k+\gamma_{uric}\sum_{i=1}^{n} (\x_i^{\prime})^2-\gamma_{uric}\nonumber \\
& = & \min_{\gamma_{uric},\nu_{uric}\geq 0}\max_{\b_i\in\{0,1\},\phi_i,\x^{\prime}} \sum_{i=1}^{n}\h_i\b_i\x_i^{\prime}
-\nu_{uric}\sum_{i=1}^{n}\b_i+\nu_{uric}k-\gamma_{uric}\sum_{i=1}^{n} (\x_i^{\prime})^2+\gamma_{uric}\nonumber \\
& = & \min_{\gamma_{uric},\nu_{uric}\geq 0} \sum_{i=1}^{n}\t_i+\nu_{uric}k+\gamma_{uric},
\label{eq:uriceq01}
\end{eqnarray}
where
\begin{equation}
\t_i=\max\{\frac{\h_i^2}{4\gamma_{uric}}-\nu_{uric},0\}.\label{eq:deftiuriclift}
\end{equation}
Positivity condition on $\nu_{uric}$ is added although it is not necessary (it essentially amount to relaxing the last constraint to an inequality which changes nothing with respect to the final results). Although we showed an inequality on $f_{uric}$ (which is sufficient for what we need here) we do mention that the above actually holds with the equality. Let
\begin{equation}
f_1^{(uric)}(\h,\gamma_{uric},\nu_{uric},\beta)=\sum_{i=1}^{n} \t_i.\label{eq:deff1uriclift}
\end{equation}
Then
\begin{multline}
\hspace{-.3in}I_{uric}(c_3^{(s)},\beta)  =  \frac{1}{nc_3^{(s)}}\log(E(\max_{\x\in\Sric}(e^{c_3^{(s)}\sqrt{n}\h^T\x}))) = \frac{1}{nc_3^{(s)}}\log(E(\max_{\x\in\Sric}(e^{c_3^{(s)}\sqrt{n}f(\x))})))\\=\frac{1}{nc_3^{(s)}}\log(Ee^{c_3^{(s)}\sqrt{n}\min_{\gamma_{uric},\nu_{uric}\geq 0}(f_1^{(uric)}(\h,\gamma_{uric},\nu_{uric},\beta)+\nu_{uric}k+\gamma_{uric})})\\
\doteq \frac{1}{nc_3^{(s)}}\min_{\gamma_{uric},\nu_{uric}\geq 0}\log(Ee^{c_3^{(s)}\sqrt{n}(f_1^{(uric)}(\h,\gamma_{uric},\nu_{uric},\beta)+\nu_{uric}k+\gamma_{uric})})\\
=\min_{\gamma_{uric},\nu_{uric}\geq 0}(\nu_{uric}\sqrt{n}\beta+ \frac{\gamma_{uric}}{\sqrt{n}}+\frac{1}{nc_3^{(s)}}\log(Ee^{c_3^{(s)}\sqrt{n}(f_1^{(uric)}(\h,\gamma_{uric},\nu_{uric},\beta))}))\\
=\min_{\gamma_{uric},\nu_{uric}\geq 0}(\nu_{uric}\sqrt{n}\beta+ \frac{\gamma_{uric}}{\sqrt{n}}+\frac{1}{nc_3^{(s)}}\log(Ee^{c_3^{(s)}(\sum_{i=1}^{n}\t_i)})),\label{eq:gamaiden1uriclift}
\end{multline}
where $\t_i$ is as given in (\ref{eq:deftiuriclift}) and as earlier, $\doteq$ stands for equality when $n\rightarrow \infty$ and would be obtained through the mechanism presented in \cite{SPH} (as discussed in \cite{SPH}, for our needs here though, even just replacing $\doteq$ with a simple $\leq$ inequality suffices). Now if one sets $\gamma_{uric}=\gamma_{uric}^{(s)}\sqrt{n}$ and $\nu_{uric}^{(s)}=\nu_{uric}\sqrt{n}$ then (\ref{eq:gamaiden1uriclift}) gives
\begin{eqnarray}
I_{uric}(c_3^{(s)},\beta)
& = & \min_{\gamma_{uric},\nu_{uric}\geq 0}(\nu_{uric}\sqrt{n}\beta+ \frac{\gamma_{uric}}{\sqrt{n}}+\frac{1}{nc_3^{(s)}}\log(Ee^{c_3^{(s)}(\sum_{i=1}^{n}\t_i)}))\nonumber \\
& = &
\min_{\gamma_{uric}^{(s)},\nu_{uric}^{(s)}\geq 0}(\nu_{uric}^{(s)}\beta+ \gamma_{uric}^{(s)}+\frac{1}{c_3^{(s)}}\log(Ee^{c_3^{(s)}\t_i^{(s)}})),\label{eq:gamaiden2uric}
\end{eqnarray}
where
\begin{equation}
\t_i^{(s)}=\max\{\frac{\h_i^2}{4\gamma_{uric}^{(s)}}-\nu_{uric}^{(s)},0\},\label{eq:deftisuriclift}
\end{equation}
or in other words
\begin{equation}
\t_i^{(s)}=\begin{cases}\frac{\h_i^2}{4\gamma_{uric}^{(s)}}-\nu_{uric}^{(s)}, & |\h_i|\geq 2\sqrt{\gamma_{uric}^{(s)}\nu_{uric}^{(s)}}\\
0, & |\h_i|\leq 2\sqrt{\gamma_{uric}^{(s)}\nu_{uric}^{(s)}}\end{cases}.
\label{eq:deftis1uriclift}
\end{equation}
The above characterization is then sufficient to compute upper bounds on $E\xi_{uric}$. However, since there is a bit of numerical work involved it is probably more convenient to look for a neater representation. That obviously involves solving several integrals. We skip such a tedious job but present the final results. We start with assuming (to insure the integrals convergence) $\gamma_{uric}^{(s)}>\frac{c_3^{(s)}}{2}$ and
setting
\begin{equation}
I^{(uric)}=Ee^{c_3^{(s)}\t_i^{(s)}}\label{eq:defbigIuric}
\end{equation}
and
\begin{eqnarray}
p_{uric} & = & c_3^{(s)}/4/\gamma_{uric}^{(s)}\nonumber \\
r_{uric} & = & -c_3^{(s)}\nu_{uric}^{(s)}\nonumber \\
C_{uric} & = & e^{r_{uric}}/\sqrt{1-2p_{uric}}\nonumber \\
a_{uric} & = & 2\sqrt{\nu_{uric}^{(s)}\gamma_{uric}^{(s)}}\sqrt{1-2p_{uric}}.\label{defhelpbigIuric}
\end{eqnarray}
Then one has
\begin{equation}
I^{(uric)}=Ee^{c_3^{(s)}\t_i^{(s)}}=C_{uric}\mbox{erfc}(a_{uric}/\sqrt{2})+(1-\mbox{erfc}(\sqrt{2\nu_{uric}^{(s)}\gamma_{uric}^{(s)}})),\label{eq:defbigIuric1}
\end{equation}
which in combination with (\ref{eq:gamaiden1uriclift}) is then enough to compute the upper bounds on $E\xi_{uric}$.

We summarize the above results related to the upper bound of $E\xi_{uric}$ in the following theorem.

\begin{theorem}($E\xi_{uric}$ - lowered upper bound)
Let $A$ be an $m\times n$ matrix with i.i.d. standard normal components. Let $k,m,n$ be large
and let $\alpha=\frac{m}{n}$ and $\beta=\frac{k}{n}$ be constants
independent of $m$ and $n$. Further, let $\Sric$ be as defined in (\ref{eq:defSric}) (or in (\ref{eq:defSriclift})). Let $\mbox{erf}$ be the standard error function associated with zero-mean unit variance Gaussian random variable and let $\mbox{erfc}=1-\mbox{erf}$.
Let
\begin{equation}
\widehat{\gamma_{sph}^{(s)}}=\frac{2c_3^{(s)}+\sqrt{4(c_3^{(s)})^2+16\alpha}}{8},\label{eq:gamasphthmuriclift}
\end{equation}
and
\begin{equation}
I_{sph}(c_3^{(s)},\alpha)=
\left ( \widehat{\gamma_{sph}^{(s)}}-\frac{\alpha}{2c_3^{(s)}}\log(1-\frac{c_3^{(s)}}{2\widehat{\gamma_{sph}^{(s)}}}\right ).\label{eq:Isphthmuriclift}
\end{equation}
Further, let $c_3^{(s)}$ an d$\gamma_{uric}^{(s)}$ be such that $\frac{c_3^{(s)}}{4\gamma_{uric}^{(s)}}<\frac{1}{2}$. Also,
let $I^{(uric)}$ be defined through (\ref{eq:defbigIuric})-(\ref{eq:defbigIuric1}) and let
\begin{equation}
I_{uric}(c_3^{(s)},\beta)=\min_{\gamma_{uric}^{(s)}\geq c_3^{(s)}/2,\nu_{uric}^{(s)}\geq 0}(\nu_{uric}^{(s)}\beta+ \gamma_{uric}^{(s)}+\frac{1}{c_3^{(s)}}\log(I^{(uric)})).\label{eq:Iuricthmuric}
\end{equation}
Then
\begin{equation}
\lim_{n\rightarrow\infty}\frac{E\xi_{uric}}{\sqrt{m}}=\lim_{n\rightarrow\infty}\frac{E(\max_{\x\in\Sric}\|A\x\|_2)}{\sqrt{m}}
\leq
\frac{1}{\sqrt{\alpha}}\min_{c_3^{(s)}\geq 0}\left (-\frac{c_3^{(s)}}{2}+I_{uric}(c_3^{(s)},\beta_{str})+I_{sph}(c_3^{(s)},\alpha)\right ).\label{eq:ubthmuriclift}
\end{equation}
Moreover, let $\xi_{uric}^{(u,low)}$ be a quantity such that
\begin{equation}
\frac{1}{\sqrt{\alpha}}\min_{c_3^{(s)}\geq 0}\left (-\frac{c_3^{(s)}}{2}+I_{uric}(c_3^{(s)},\beta_{str})+I_{sph}(c_3^{(s)},\alpha)\right )
<\xi_{uric}^{(u,low)}.
\label{eq:conduricthm}
\end{equation}
Then
\begin{eqnarray}
& & \lim_{n\rightarrow\infty}P(\max_{\x\in\Sric}(\|A\x\|_2)\leq \xi_{uric}^{(u,low)}\sqrt{m})\geq 1\nonumber \\
& \Leftrightarrow & \lim_{n\rightarrow\infty}P(\xi_{uric}\leq \xi_{uric}^{(u,low)}\sqrt{m})\geq 1 \nonumber \\
& \Leftrightarrow & \lim_{n\rightarrow\infty}P(\xi_{uric}^2\leq (\xi_{uric}^{(u,low)})^2 m)\geq 1. \label{eq:uricprobthm}
\end{eqnarray}
\label{thm:thmuriclift}
\end{theorem}
\begin{proof}
The first part follows from the above discussion. The moreover part follows from considerations presented in \cite{StojnicCSetam09,StojnicHopBnds10,StojnicMoreSophHopBnds10}.
\end{proof}

We will below present the results one can get using the above theorem. However, before proceeding with the discussion of the results one can obtain through  Theorem \ref{thm:thmuriclift}, we also mention that the results presented in the previous section (essentially in Lemma \ref{lemma:uriclemma}) can in fact be deduced from the above theorem. Namely, in the limit $c_3^{(s)}\rightarrow 0$, one from (\ref{eq:chpos8uriclift}) has that $E\max_{\x\in\Sric}\h^T\x+\sqrt{\alpha n}$ can be used as an upper bound on $E\xi_{uric}$. This is of course exactly the same expression that was considered in the previous section. For the completeness we present the following corollary where we actually derive the results from the previous section as a special case of those given in the above theorem (of course, the special case actually assumes  $c_3^{(s)}\rightarrow 0$).

\begin{corollary}($E\xi_{uric}$ - upper bound)
Assume the setup of Theorem \ref{thm:thmuriclift}. Let $c_3^{(s)}\rightarrow 0$. Then
\begin{equation}
\widehat{\gamma_{sph}^{(s)}}\rightarrow \frac{\sqrt{\alpha}}{2},\label{eq:gamasphcoruric}
\end{equation}
and
\begin{equation}
I_{sph}(c_3^{(s)},\alpha)\rightarrow
\sqrt{\alpha}.\label{eq:Isphcoruric}
\end{equation}
Further, $p_{uric}\rightarrow 0$ and set $\nu^2=4\nu_{uric}^{(s)}\gamma_{uric}^{(s)}$
\begin{eqnarray}
C_{uric} & \rightarrow & 1-c_3^{(s)}\nu_{uric}^{(s)}+\frac{c_3^{(s)}}{4\gamma_{uric}^{(s)}}=1+c_3^{(s)}\frac{1-\nu^2}{4\gamma_{uric}^{(s)}}\nonumber \\
I^{(uric)}
& \rightarrow &   C_{uric}\mbox{erfc}(\sqrt{2\nu_{uric}^{(s)}\gamma_{uric}^{(s)}}\sqrt{1-2p_{uric}})+(1-\mbox{erfc}(\sqrt{2\nu_{uric}^{(s)}\gamma_{uric}^{(s)}}))\nonumber \\
& \rightarrow & 1+c_3^{(s)}\frac{1-\nu^2}{4\gamma_{uric}^{(s)}}\mbox{erfc}(\sqrt{\nu^2/2}\sqrt{1-2p_{uric}})
+\mbox{erfc}(\sqrt{\nu^2/2}\sqrt{1-2p_{uric}})
-\mbox{erfc}(\sqrt{\nu^2/2})\nonumber \\
& \rightarrow &
1+c_3^{(s)}\frac{1-\nu^2}{4\gamma_{uric}^{(s)}}\mbox{erfc}(\nu/\sqrt{2})
+\frac{2}{\sqrt{\pi}}\frac{\nu}{\sqrt{2}}e^{-\frac{\nu^2}{2}}\frac{c_3^{(s)}}{4\gamma_{uric}^{(s)}}.\nonumber \\\label{eq:defI1I2uriccoruric}
\end{eqnarray}
Moreover, let
\begin{eqnarray}
\hspace{-.4in}I_{uric}(c_3^{(s)},\beta) & \rightarrow & \min_{\gamma_{uric}^{(s)},\nu\geq 0}
\left (\frac{\nu^2\beta}{4\gamma_{uric}^{(s)}}+\gamma_{sec}^{(s)}+\frac{(1-\nu^2)\mbox{erfc}(\nu/\sqrt{2})+
\frac{2}{\sqrt{\pi}}\frac{\nu}{\sqrt{2}}e^{-\frac{\nu^2}{2}}}{4\gamma_{sec}^{(s)}}\right )\nonumber \\
& = & \min_{\nu\geq 0}
\sqrt{\left (\beta\nu^2+
\mbox{erfc}(\frac{\nu}{\sqrt{2}})(1-\nu^2)+\frac{2\nu e^{-\frac{\nu^2}{2}}}{\sqrt{2\pi}})\right )}.\label{eq:Iuriccoruric}
\end{eqnarray}
Choosing $\nu=\sqrt{2}\mbox{erfinv}(1-\beta)$
one then has
\begin{equation}
\lim_{n\rightarrow \infty}\frac{E\xi_{uric}}{\sqrt{m}}=\frac{E(\max_{\x\in\Sric} \|A\x\|_2)}{\sqrt{m}} \leq 1+\frac{1}{\sqrt{\alpha}}\sqrt{\beta+\frac{2\mbox{erfinv}(1-\beta)}{\sqrt{\pi}e^{(\mbox{erfinv}(1-\beta))^2}}}.\label{eq:uricexpcor}
\end{equation}
\label{cor:coruric}
\end{corollary}
\begin{proof}
Theorem \ref{thm:thmuriclift} holds for any $c_3^{(s)}\geq 0$. The above corollary instead of looking for the best possible $c_3^{(s)}$ in Theorem \ref{thm:thmuriclift} assumes a simple $c_3^{(s)}\rightarrow 0$ scenario. The proof of the fact that in such a scenario the upper bounds formulation given in Theorem \ref{thm:thmuriclift} indeed boils down to what is stated in Lemma \ref{lemma:uriclemma} is essentially contained in the steps mentioned above. The choice for $\nu$ is actually optimal (however, we skip showing that).

Alternatively, as mentioned above, one can look at $\frac{E\max_{\x\in\Sric}\h^T\x}{\sqrt{m}}+1$ and following the methodology presented in (\ref{eq:uriceq01}) (and originally in \cite{StojnicCSetam09}) obtain for a scalar $\nu=\sqrt{2}\mbox{erfinv}(1-\beta)$
\begin{equation}
\frac{E\max_{\x\in\Sric}\h^T\x}{\sqrt{m}} +1 \leq \frac{1}{\sqrt{\alpha}}\sqrt{E_{\nu\leq |\h_i|}|\h_i|^2} +1.\label{eq:coruric0}
\end{equation}
Solving the integral (and using all the concentrating machinery of \cite{StojnicCSetam09}) one can write
\begin{equation}
\frac{E\max_{\x\in\Sric}\h^T\x}{\sqrt{m}}+1  \doteq  \frac{1}{\sqrt{\alpha}}\sqrt{\left (\int_{\nu\leq |\h_i|}|\h_i|^2\frac{e^{-\frac{\h_i^2}{2}}d\h_i}{\sqrt{2\pi}}\right )}+1= \frac{1}{\sqrt{\alpha}} \sqrt{\left (\mbox{erfc}(\frac{\nu}{\sqrt{2}})+\frac{2\nu e^{-\frac{\nu^2}{2}}}{\sqrt{2\pi}}\right )}+1.\label{eq:coruric1}
\end{equation}
Connecting beginning and end in (\ref{eq:coruric1}) then leads to the condition given in the above corollary.
\end{proof}


\subsection{Numerical results -- lowered upper bound on $\xi_{uric}$}
\label{sec:xiuprobnumlift}

In this subsection we present a small collection of numerical results one can obtain based on Theorem \ref{thm:thmuriclift}. In Tables \ref{tab:uriclifttab1det} and \ref{tab:uriclifttab2det} we show the upper bounds on $\lim_{n\rightarrow\infty}\frac{E\xi_{uric}}{\sqrt{m}}$ one can obtain based on Theorem \ref{thm:thmuriclift}. We refer to those bounds as $\xi_{uric}^{(u,low)}$. Also, to get a feeling how the results of Theorem \ref{thm:thmuriclift} fare when compared to the ones presented in the previous section we in Tables \ref{tab:uriclifttab1} and \ref{tab:uriclifttab2} also present the results we obtained in Subsection \ref{sec:xiuprobnum} (which are of course based on Lemma \ref{lemma:uriclemma}  and Corollary \ref{cor:coruric}). For completeness, we in Tables \ref{tab:uriclifttab1} and \ref{tab:uriclifttab2} also recall on the results from \cite{BT10}.

\begin{table}
\caption{Lowered upper bounds on $\lim_{n\rightarrow\infty}\frac{E\xi_{uric}}{\sqrt{m}}$ -- low $\beta/\alpha\leq 0.5$ regime; optimized parameters}\vspace{.1in}
\hspace{-0in}\centering
\begin{tabular}{||l|c|c|c|c|c||}\hline\hline
 \hspace{.5in}$\alpha$                                                        & $0.1$ & $0.3$ & $0.5$ & $0.7$ & $0.9$  \\ \hline\hline
 $\beta/\alpha=0.1$; $c_{3}^{(s)}$                                            & $0.2577$ & $0.3596$ & $0.4033$ & $0.4247$ & $0.4338$ \\ \hline
 $\beta/\alpha=0.1$; $\nu_{uric}^{(s)}$                                       & $11.375$ & $5.4640$ & $3.7775$ & $2.9153$ & $2.3745$ \\ \hline
 $\beta/\alpha=0.1$; $\gamma_{uric}^{(s)}$                                    & $0.1866$ & $0.2837$ & $0.3388$ & $0.3773$ & $0.4063$ \\ \hline
 $\beta/\alpha=0.1$; $\xi_{uric}^{(u,low)}$                                   & $1.8525$ & $1.7602$ & $1.7129$ & $1.6798$ & $1.6538$ \\ \hline\hline
 $\beta/\alpha=0.3$; $c_{3}^{(s)}$                                            & $0.2893$ & $0.3448$ & $0.3336$ & $0.3005$ & $0.2584$ \\ \hline
 $\beta/\alpha=0.3$; $\nu_{uric}^{(s)}$                                       & $5.4820$ & $2.3578$ & $1.4759$ & $1.0278$ & $0.7494$ \\ \hline
 $\beta/\alpha=0.3$; $\gamma_{uric}^{(s)}$                                    & $0.2675$ & $0.3854$ & $0.4409$ & $0.4721$ & $0.4900$ \\ \hline
 $\beta/\alpha=0.3$; $\xi_{uric}^{(u,low)}$                                   & $2.3338$ & $2.1409$ & $2.0386$ & $1.9650$ & $1.9061$ \\ \hline\hline
 $\beta/\alpha=0.5$; $c_{3}^{(s)}$                                            & $0.2833$ & $0.2914$ & $0.2386$ & $0.1748$ & $0.1157$ \\ \hline
 $\beta/\alpha=0.5$; $\nu_{uric}^{(s)}$                                       & $3.7653$ & $1.4663$ & $0.8237$ & $0.5036$ & $0.3121$ \\ \hline
 $\beta/\alpha=0.5$; $\gamma_{uric}^{(s)}$                                    & $0.3117$ & $0.4313$ & $0.4771$ & $0.4961$ & $0.5026$ \\ \hline
 $\beta/\alpha=0.5$; $\xi_{uric}^{(u,low)}$                                   & $2.6190$ & $2.3437$ & $2.1948$ & $2.0868$ & $2.0005$ \\ \hline\hline
\end{tabular}
\label{tab:uriclifttab1det}
\end{table}

\begin{table}
\caption{Lowered upper bounds on $\lim_{n\rightarrow\infty}\frac{E\xi_{uric}}{\sqrt{m}}$ -- low $\beta/\alpha> 0.5$ regime; optimized parameters}\vspace{.1in}
\hspace{-0in}\centering
\begin{tabular}{||l|c|c|c|c|c||}\hline\hline
 \hspace{.5in}$\alpha$                                                        & $0.1$ & $0.3$ & $0.5$ & $0.7$ & $0.9$  \\ \hline\hline
 $\beta/\alpha=0.7$; $c_{3}^{(s)}$                                            & $0.2694$ & $0.2379$ & $0.1589$ & $0.0869$ & $0.0365$ \\ \hline
 $\beta/\alpha=0.7$; $\nu_{uric}^{(s)}$                                       & $2.8847$ & $1.0152$ & $0.5014$ & $0.2557$ & $0.1195$ \\ \hline
 $\beta/\alpha=0.7$; $\gamma_{uric}^{(s)}$                                    & $0.3425$ & $0.4577$ & $0.4923$ & $0.5015$ & $0.5020$ \\ \hline
 $\beta/\alpha=0.7$; $\xi_{uric}^{(u,low)}$                                   & $2.8268$ & $2.4774$ & $2.2863$ & $2.1481$ & $2.0392$ \\ \hline\hline
 $\beta/\alpha=0.9$; $c_{3}^{(s)}$                                            & $0.2535$ & $0.1898$ & $0.0982$ & $0.0341$ & $0.0051$ \\ \hline
 $\beta/\alpha=0.9$; $\nu_{uric}^{(s)}$                                       & $2.3337$ & $0.7375$ & $0.3103$ & $0.1193$ & $0.0290$ \\ \hline
 $\beta/\alpha=0.9$; $\gamma_{uric}^{(s)}$                                    & $0.3659$ & $0.4740$ & $0.4984$ & $0.5014$ & $0.5004$ \\ \hline
 $\beta/\alpha=0.9$; $\xi_{uric}^{(u,low)}$                                   & $2.9907$ & $2.5723$ & $2.3426$ & $2.1784$ & $2.0522$  \\ \hline\hline
\end{tabular}
\label{tab:uriclifttab2det}
\end{table}

As can be seen from the tables, while conceptually substantial, in practice the improvement lowered bounds from Theorem \ref{thm:thmuriclift} provide may not always be significant. That can be because the methods are not powerful enough to make a bigger improvement or simply because a big improvement may not be possible (in other words the results obtained in Lemma \ref{lemma:uriclemma} may very well already be fairly close to the optimal ones). As for the limits of the developed methods, we do want to emphasize that we did solve the numerical optimizations that appear in Theorem \ref{thm:thmuriclift} only on a local optimum level and obviously only with a finite precision. We do not know if a substantial change would occur in the presented results had we solved it on a global optimum level (we recall that finding local optima is of course certainly enough to establish valid upper bounds; moreover in Tables \ref{tab:uriclifttab1det} and \ref{tab:uriclifttab2det} we provide a detailed values for optimizing parameters that we chose). As for how far away from the true $E\xi_{uric}$ are the results presented in the tables, we actually believe that they are in fact very close to the optimal ones.

\begin{table}
\caption{Lowered upper bounds on $\lim_{n\rightarrow\infty}\frac{E\xi_{uric}}{\sqrt{m}}$ -- low $\beta/\alpha\leq 0.5$ regime}\vspace{.1in}
\hspace{-0in}\centering
\begin{tabular}{||l|c|c|c|c|c||}\hline\hline
 \hspace{1in}$\alpha$                                                         & $0.1$ & $0.3$ & $0.5$ & $0.7$ & $0.9$  \\ \hline\hline
 $\beta/\alpha=0.1$; $\xi_{uric}^{BT}$                                        & $1.9786$ & $1.8970$ & $1.8562$ & $1.8280$ & $1.8062$ \\ \hline
 $\beta/\alpha=0.1$; $\xi_{uric}^{(u)}$ ($c_3^{(s)}\rightarrow 0$)            & $1.9192$ & $1.8049$ & $1.7471$ & $1.7071$ & $1.6761$ \\ \hline
 $\beta/\alpha=0.1$; $\xi_{uric}^{(u,low)}$ (optimized $c_3^{(s)}$)           & $1.8525$ & $1.7602$ & $1.7129$ & $1.6798$ & $1.6538$ \\ \hline\hline
 $\beta/\alpha=0.3$; $\xi_{uric}^{BT}$                                        & $2.5822$ & $2.4067$ & $2.3142$ & $2.2471$ & $2.1925$ \\ \hline
 $\beta/\alpha=0.3$; $\xi_{uric}^{(u)}$ ($c_3^{(s)}\rightarrow 0$)            & $2.3941$ & $2.1710$ & $2.0560$ & $1.9753$ & $1.9123$ \\ \hline
 $\beta/\alpha=0.3$; $\xi_{uric}^{(u,low)}$ (optimized $c_3^{(s)}$)           & $2.3338$ & $2.1409$ & $2.0386$ & $1.9650$ & $1.9061$ \\ \hline\hline
 $\beta/\alpha=0.5$; $\xi_{uric}^{BT}$                                        & $2.9622$ & $2.7036$ & $2.5591$ & $2.4479$ & $2.3508$ \\ \hline
 $\beta/\alpha=0.5$; $\xi_{uric}^{(u)}$ ($c_3^{(s)}\rightarrow 0$)            & $2.6706$ & $2.3633$ & $2.2030$ & $2.0901$ & $2.0017$ \\ \hline
 $\beta/\alpha=0.5$; $\xi_{uric}^{(u,low)}$ (optimized $c_3^{(s)}$)           & $2.6190$ & $2.3437$ & $2.1948$ & $2.0868$ & $2.0005$ \\ \hline\hline
\end{tabular}
\label{tab:uriclifttab1}
\end{table}

\begin{table}
\caption{Lowered upper bounds on $\lim_{n\rightarrow\infty}\frac{E\xi_{uric}}{\sqrt{m}}$ -- high $\beta/\alpha> 0.5$ regime}\vspace{.1in}
\hspace{-0in}\centering
\begin{tabular}{||l|c|c|c|c|c||}\hline\hline
 \hspace{1in}$\alpha$                                                         & $0.1$ & $0.3$ & $0.5$ & $0.7$ & $0.9$  \\ \hline\hline
 $\beta/\alpha=0.7$; $\xi_{uric}^{BT}$                                        & $3.2505$ & $2.9094$ & $2.7053$ & $2.5337$ & $2.3769$ \\ \hline
 $\beta/\alpha=0.7$; $\xi_{uric}^{(u)}$  ($c_3^{(s)}\rightarrow 0$)           & $2.8709$ & $2.4898$ & $2.2898$ & $2.1489$ & $2.0394$ \\ \hline
 $\beta/\alpha=0.7$; $\xi_{uric}^{(u,low)}$  (optimized $c_3^{(s)}$)          & $2.8268$ & $2.4774$ & $2.2863$ & $2.1481$ & $2.0392$ \\ \hline\hline
 $\beta/\alpha=0.9$; $\xi_{uric}^{BT}$                                        & $3.4849$ & $3.0577$ & $2.7779$ & $2.5385$ & $2.3769$ \\ \hline
 $\beta/\alpha=0.9$; $\xi_{uric}^{(u)}$  ($c_3^{(s)}\rightarrow 0$)           & $3.0283$ & $2.5801$ & $2.3440$ & $2.1785$ & $2.0522$ \\ \hline
 $\beta/\alpha=0.9$; $\xi_{uric}^{(u,low)}$  (optimized $c_3^{(s)}$)          & $2.9907$ & $2.5723$ & $2.3426$ & $2.1784$ & $2.0522$ \\ \hline\hline
\end{tabular}
\label{tab:uriclifttab2}
\end{table}

\section{Bounding $\xi_{lric}$}
\label{sec:xil}

In this section we look at $\xi_{lric}$ and design a mechanism that can be used to lower-bound it. The mechanism will be an appropriate adaption of the mechanism presented in Section \ref{sec:xiu} (clearly, as such it will be to an extent related to the mechanism we presented in \cite{StojnicCSetam09} and used for the analysis of (\ref{eq:l1})'s ability to recover $\tilde{\x}$). As earlier, we will again assume a substantial level of familiarity with many of the well-known results that relate to the performance characterization of (\ref{eq:l1}). We start by recalling on the definition of set $\Sric$ from (\ref{eq:defSric})
\begin{equation}
\Sric=\{\x\in S^{n-1}| \quad \|\x\|_{\ell_0}=k\},\label{eq:defSricl}
\end{equation}
where $S^{n-1}$ is the unit sphere in $R^n$. Then one can transform the second part of (\ref{eq:ulric}) in the following way
\begin{equation}
\xi_{lric}=\min_{\x\in\Sric}\|A\x\|_2.\label{eq:negham1l}
\end{equation}
As mentioned in Section \ref{sec:xiu}, a set of problems very similar to (\ref{eq:negham1l}) was considered in \cite{StojnicCSetam09,StojnicHopBnds10}. We will here utilize mechanisms similar to some of those from \cite{StojnicCSetam09,StojnicHopBnds10} and will attempt to establish a set of lower bounds on $\xi_{lric}$. However, as was the case in Section \ref{sec:xiu}, one should note that the structure of set $\Sric$ is somewhat different than the structure of sets considered in \cite{StojnicCSetam09,StojnicHopBnds10} and again a careful approach will be needed to readapt the mechanisms from \cite{StojnicCSetam09,StojnicHopBnds10} to the problem we consider here. Also, as earlier, since the mechanisms of \cite{StojnicCSetam09,StojnicHopBnds10} were powerful enough to establish the concentration of quantities similar to $\xi_{lric}$ we will mostly focus only on $E\xi_{lric}$. Below we present a way to create a lower-bound on the optimal value of $E\xi_{lric}$.

\subsection{Probabilistic approach to upper bounding $\xi_{lric}$}
\label{sec:xilprob}

In this section we look at $E\xi_{lric}$ and design its a lower-bound. To do so we rely on the following lemma (which is a modified version of a similar lemma from \cite{StojnicHopBnds10} and, as mentioned in \cite{StojnicHopBnds10}, a direct application of Theorem $2$ from \cite{StojnicHopBnds10} proven in \cite{Gordon85}):
\begin{lemma}
Let $A$ be an $m\times n$ matrix with i.i.d. standard normal components. Let $\g$ and $\h$ be $n\times 1$ and $m\times 1$ vectors, respectively, with i.i.d. standard normal components. Also, let $g$ be a standard normal random variable. Then
\begin{equation}
E(\min_{\x\in\Sric}\max_{\|\y\|_2=1}(\y^T A\x +\|\x\|_2 g))\geq E(\min_{\x\in\Sric}\max_{\|\y\|_2=1}(\|\x\|_2\g^T\y+\h^T\x)).\label{eq:posexplemmal}
\end{equation}\label{lemma:posexplemmal}
\end{lemma}
\begin{proof}
As mentioned above, the proof is a standard/direct application of Theorem $2$ from \cite{StojnicHopBnds10} proven in \cite{Gordon85}). We skip the details and mention that, as in Lemma \ref{lemma:uriclemma}, the only difference between the proof one needs here and the corresponding one given in \cite{StojnicHopBnds10} is the structure of set $\Sric$. However, such a difference changes nothing in the remainder of the proof.
\end{proof}

Using results of Lemma \ref{lemma:posexplemma} we then have
\begin{multline}
E(\min_{\x\in\Sric} \|A\x\|_2) =E(\min_{\x\Sric}\max_{\|\y\|_2=1}(\y^T A\x +\|\x\|_2 g))\\\geq E(\min_{\x\in\Sric}\max_{\|\y\|_2=1}(\|\x\|_2\g^T\y+\h^T\x))=E\|\x\|_2\|\g\|_2+E\min_{\x\in\Sric}\h^T\x\geq \sqrt{m}-\frac{1}{4\sqrt{m}}+E\min_{\x\in\Sric}\h^T\x.\label{eq:poshopaftlemma2l}
\end{multline}
Let $\bar{\h}$ be the vector of magnitudes of $\h$ sorted in nondecreasing order (of course, ties are broken arbitrarily). Then from (\ref{eq:poshopaftlemma2}) we have
\begin{equation}
E(\min_{\x\in\Sric} \|A\x\|_2)\geq \sqrt{m}-\frac{1}{4\sqrt{m}}-E\sqrt{\sum_{i=n-k+1}^{n}\bar{\h}_i}\geq \sqrt{m}-\frac{1}{4\sqrt{m}}-\sqrt{E\sum_{i=n-k+1}^{n}\bar{\h}_i}.\label{eq:poshopaftlemma3l}
\end{equation}
Using the results of \cite{StojnicCSetam09} one then has
\begin{equation}
\lim_{n\rightarrow \infty}\frac{E(\min_{\x\in\Sric} \|A\x\|_2)}{\sqrt{m}}\geq 1-\sqrt{\lim_{n\rightarrow \infty}\frac{E\sum_{i=n-k+1}^{n}\bar{\h}_i^2}{\alpha n}}= 1-\frac{1}{\sqrt{\alpha}}\sqrt{\beta+\frac{2\mbox{erfinv}(1-\beta)}{\sqrt{\pi}e^{(\mbox{erfinv}(1-\beta))^2}}}.\label{eq:poshopaftlemma4l}
\end{equation}
Connecting beginning and end of (\ref{eq:poshopaftlemma4l}) we finally have an upper bound on $E\xi_{lric}$ (in a scaled more appropriate form),
\begin{equation}
\lim_{n\rightarrow \infty}\frac{E\xi_{lric}}{\sqrt{m}}=\frac{E(\min_{\x\in\Sric} \|A\x\|_2)}{\sqrt{m}} \geq 1-\frac{1}{\sqrt{\alpha}}\sqrt{\beta+\frac{2\mbox{erfinv}(1-\beta)}{\sqrt{\pi}e^{(\mbox{erfinv}(1-\beta))^2}}}.\label{eq:poshopubexp1l}
\end{equation}

We summarize our results from this subsection in the following lemma.

\begin{lemma}
Let $A$ be an $m\times n$ matrix with i.i.d. standard normal components. Let $n$ be large and let $k=\beta n$, $m=\alpha n$, where $\beta,\alpha>0$ are constants independent of $n$. Let $\xi_{lric}$ be as in (\ref{eq:negham1l}).
\begin{equation}
\lim_{n\rightarrow \infty}\frac{E\xi_{lric}}{\sqrt{m}}=\frac{E(\min_{\x\in\Sric} \|A\x\|_2)}{\sqrt{m}} \geq 1-\frac{1}{\sqrt{\alpha}}\sqrt{\beta+\frac{2\mbox{erfinv}(1-\beta)}{\sqrt{\pi}e^{(\mbox{erfinv}(1-\beta))^2}}}.\label{eq:lricexplemma}
\end{equation}
Moreover, let $\xi_{lric}^{(l)}$ be a quantity such that
\begin{equation}
1-\frac{1}{\sqrt{\alpha}}\sqrt{\beta+\frac{2\mbox{erfinv}(1-\beta)}{\sqrt{\pi}e^{(\mbox{erfinv}(1-\beta))^2}}}>\xi_{lric}^{(l)}.
\label{eq:condlriclemma}
\end{equation}
Then
\begin{eqnarray}
& & \lim_{n\rightarrow\infty}P(\min_{\x\in\Sric}(\|A\x\|_2)\geq \xi_{lric}^{(l)}\sqrt{m})\geq 1\nonumber \\
& \Leftrightarrow & \lim_{n\rightarrow\infty}P(\xi_{lric}\geq \xi_{lric}^{(l)}\sqrt{m})\geq 1 \nonumber \\
& \Leftrightarrow & \lim_{n\rightarrow\infty}P(\xi_{lric}^2\geq (\xi_{lric}^{(l)})^2 m)\geq 1. \label{eq:lricproblemma}
\end{eqnarray}
\label{lemma:lriclemma}
\end{lemma}
\begin{proof}
As was the case with the proof of Lemma \ref{lemma:uriclemma}, the proof of (\ref{eq:lricexplemma}) follows from (\ref{eq:poshopubexp1l}) and the above discussion. The proof of the moreover part follows from the concentration properties considered in \cite{StojnicCSetam09} and the corresponding discussion presented in \cite{StojnicHopBnds10}.
\end{proof}

\noindent \textbf{Remark:} Of course, the above lower bounds may occasionally fall below zero. In that case they would be trivially useless. However, instead of formally replacing them with zero when that happens we purposely leave them in the above form to emphasize their potential deficiency.

\subsection{Numerical results -- lower bound on $\xi_{lric}$}
\label{sec:xilprobnum}

Similarly to what we did in Section \ref{sec:xiuprobnum}, in this subsection we present a small collection of numerical results one can obtain based on Lemma \ref{lemma:lriclemma}. In Table \ref{tab:lrictab1} we essentially show the lower bounds on $\lim_{n\rightarrow\infty}\frac{E\xi_{lric}}{\sqrt{m}}$ one can obtain based on the above lemma. We refer to those bounds as $\xi_{lric}^{(l)}$. Also, to get a feeling how far off they could be from the optimal ones we show a set of known lower bounds from \cite{BT10} (alternatively one can also look at the results from e.g. \cite{CRT,CT,BCTsharp09} as well; the results from \cite{BT10} however provide higher values of the lower bounds). We also point out that, as was the case when we studied $\xi_{uric}$ in Section \ref{sec:xiu}, based on numerical experiments conducted in \cite{BT10}, the lower bounds presented there appeared as if not that far away from the optimal values. We denote the upper bounds on $\lim_{n\rightarrow\infty}\frac{E\xi_{lric}}{\sqrt{m}}$ that one can obtain based on \cite{BT10} as $\xi_{lric}^{BT}$. Also, as was the case in Section \ref{sec:xiuprobnum}, the values presented in Table \ref{tab:lrictab1} are slightly modified versions of the corresponding quantities from \cite{BT10}. Namely, to get a complete agreement with \cite{BT10} one should think of ${\cal L}$ in \cite{BT10} as $1-(\xi_{lric}^{BT})^2$ (or in other words, what we call $\xi_{lric}^{BT}$ in \cite{BT10} is called $\lambda^{min}$). Overall, the results obtained based on Lemma \ref{lemma:lriclemma} are not as good as those from \cite{BT10} in a wide range of values for $\beta$ and $\alpha$. In fact, as $\beta$ gets larger the lower bounds the above lemma provides become even negative. However, the bounds given in the above lemma are relatively simple and can be used for a quick assessment of $E\xi_{lric}$ when they are positive.

\begin{table}
\caption{Lower bounds on $\lim_{n\rightarrow\infty}\frac{E\xi_{lric}}{\sqrt{m}}$ -- low $\beta/\alpha\leq 0.5$ regime}\vspace{.1in}
\hspace{-0in}\centering
\begin{tabular}{||l|c|c|c|c|c||}\hline\hline
 $\hspace{.5in}\alpha$                        & $0.1$ & $0.3$ & $0.5$ & $0.7$ & $0.9$  \\ \hline\hline
 $\beta/\alpha=0.05$; $\xi_{lric}^{BT}$       & $0.4224$ & $0.4545$ & $0.4709$ & $0.4823$ & $0.4911$ \\ \hline
 $\beta/\alpha=0.05$; $\xi_{lric}^{(l)}$      & $0.3031$ & $0.3789$ & $0.4168$ & $0.4429$ & $0.4631$ \\ \hline\hline
 $\beta/\alpha=0.1$; $\xi_{lric}^{BT}$        & $0.2717$ & $0.3120$ & $0.3335$ & $0.3489$ & $0.3611$ \\ \hline
 $\beta/\alpha=0.1$; $\xi_{lric}^{(l)}$       & $0.0808$ & $0.1951$ & $0.2529$ & $0.2929$ & $0.3239$ \\ \hline\hline
 $\beta/\alpha=0.3$; $\xi_{lric}^{BT}$        & $0.0488$ & $0.0803$ & $0.1025$ & $0.1215$ & $0.1389$ \\ \hline
 $\beta/\alpha=0.3$; $\xi_{lric}^{(l)}$       & $-0.394$ & $-0.171$ & $-0.056$ & $0.0247$ & $0.0877$ \\ \hline\hline
 $\beta/\alpha=0.5$; $\xi_{lric}^{BT}$        & $0.0041$ & $0.0130$ & $0.0234$ & $0.0356$ & $0.0504$ \\ \hline
 $\beta/\alpha=0.5$; $\xi_{lric}^{(l)}$       & $-0.670$ & $-0.363$ & $-0.203$ & $-0.090$ & $-0.002$ \\ \hline\hline
\end{tabular}
\label{tab:lrictab1}
\end{table}

\section{Lifting $\xi_{lric}$'s bounds}
\label{sec:xillift}

In the previous section we adapted the method from Section \ref{sec:xiu} for estimating $\xi_{uric}$ attempting to get good estimates for $\xi_{uric}$. However, while the method from Section \ref{sec:xiu} is very powerful when it comes to providing upper bounds on $\xi_{uric}$ it is significantly less successful when it comes to obtaining lower bounds on $\xi_{lric}$. As could have been seen from the numerical results given in the previous section, not only are the lower bounds on $\xi_{lric}$ obtained there weaker than known ones, they fairly often end up being negative. In this section we will attempt to improve the mechanisms presented in the previous section. Namely, we will attempt to adapt the strategy of Section \ref{sec:xiulow} and use some of
the ideas we recently introduced in \cite{StojnicLiftStrSec13,StojnicMoreSophHopBnds10} to provide a substantial conceptual improvement in the bounds given in Section \ref{sec:xil}. It will turn out that the improvements won't be only conceptual. In other words, the methodology that we will present below will be capable of providing significantly better practical estimates for $E\xi_{lric}$.

\subsection{Probabilistic approach to lifting $\xi_{lric}$'s bounds}
\label{sec:xilliftprob}

As in Subsection \ref{sec:xiulowprob}, we start by introducing a lemma very similar to the one considered in \cite{StojnicMoreSophHopBnds10} (the lemma is essentially a direct consequence/application of Theorem $2$ from \cite{StojnicMoreSophHopBnds10} which of course was proved in \cite{Gordon85}).
\begin{lemma}
Let $A$ be an $m\times n$ matrix with i.i.d. standard normal components. Let $\g$ and $\h$ be $n\times 1$ and $m\times 1$ vectors, respectively, with i.i.d. standard normal components. Also, let $g$ be a standard normal random variable and let $c_3$ be a positive constant. Then
\begin{equation}
E(\max_{\x\in\Sric}\min_{\|\y\|_2=1}e^{-c_3(\y^T A\x + g)})\leq E(\max_{\x\in\Sric}\min_{\|\y\|_2=1}e^{-c_3(\g^T\y+\h^T\x)}).\label{eq:posexplemmallift}
\end{equation}\label{lemma:posexplemmallift}
\end{lemma}
\begin{proof}
As mentioned above, the proof is a standard/direct application of Theorem $2$ from \cite{StojnicMoreSophHopBnds10} which was proved in \cite{Gordon85}. The only difference is the structure of $\Sric$ which changes nothing in the proof.
\end{proof}

Following what was done in \cite{StojnicMoreSophHopBnds10} one then has
\begin{equation}
E(\min_{\x\in\Sric}\|A\x\|_2)\geq
\frac{c_3}{2}-\frac{1}{c_3}\log(E(\max_{\x\in\Sric}(e^{-c_3\h^T\x})))
-\frac{1}{c_3}\log(E(\max_{\|\y\|_2=1}(e^{-c_3\g^T\y}))).\label{eq:chpos8lriclift}
\end{equation}
Let $c_3=c_3^{(s)}\sqrt{n}$ where $c_3^{(s)}$ is a constant independent of $n$. Then following further what we did in \cite{StojnicMoreSophHopBnds10} we have
\begin{equation}
\frac{E(\min_{\x\in\Sric}\|A\x\|_2)}{\sqrt{n}}
 \geq
\frac{c_3^{(s)}}{2}-\frac{1}{nc_3^{(s)}}\log(E(\max_{\x\in\Sric}(e^{-c_3^{(s)}\sqrt{n}\h^T\x})))
-\frac{1}{nc_3^{(s)}}\log(E(\max_{\|\y\|_2=1}(e^{-c_3^{(s)}\sqrt{n}\g^T\y}))),
\end{equation}
or written slightly differently
\begin{eqnarray}
\sqrt{\alpha}\frac{E(\min_{\x\in\Sric}\|A\x\|_2)}{\sqrt{m}}
& \geq &
\frac{c_3^{(s)}}{2}-\frac{1}{nc_3^{(s)}}\log(E(\max_{\x\in\Sric}(e^{-c_3^{(s)}\sqrt{n}\h^T\x})))
-\frac{1}{nc_3^{(s)}}\log(E(\max_{\|\y\|_2=1}(e^{-c_3^{(s)}\sqrt{n}\g^T\y}))) \nonumber \\
& = & -(-\frac{c_3^{(s)}}{2}+I_{lric}(c_3^{(s)},\beta)
+I_{sph}(c_3^{(s)},\alpha)),\label{eq:chpos9lriclift}
\end{eqnarray}
where
\begin{eqnarray}
I_{lric}(c_3^{(s)},\beta) & = & \frac{1}{nc_3^{(s)}}\log(E(\max_{\x\in\Sric}(e^{-c_3^{(s)}\sqrt{n}\h^T\x})))\nonumber \\
I_{sph}(c_3^{(s)},\alpha) & = & \frac{1}{nc_3^{(s)}}\log(E(\min_{\|\y\|_2=1}(e^{-c_3^{(s)}\sqrt{n}\g^T\y}))).\label{eq:defIsl}
\end{eqnarray}
In \cite{StojnicMoreSophHopBnds10} we also established the following
\begin{equation}
I_{sph}(c_3^{(s)},\alpha)=\frac{1}{nc_3^{(s)}}\log(Ee^{-c_3^{(s)}\sqrt{n}\|\g\|_2})
\doteq\widehat{\gamma^{(s)}}-\frac{\alpha}{2c_3^{(s)}}\log(1-\frac{c_3^{(s)}}{2\widehat{\gamma^{(s)}}})),\label{eq:gamaiden2l}
\end{equation}
where as in Section \ref{sec:xiulow} (and following \cite{SPH}) $\doteq$ stands for equality that holds as $n\rightarrow \infty$ and
\begin{equation}
\widehat{\gamma^{(s)}}=\frac{2c_3^{(s)}-\sqrt{4(c_3^{(s)})^2+16\alpha}}{8}.\label{eq:gamaiden3l}
\end{equation}
As in Section \ref{sec:xiulow}, we also mention that (as in \cite{StojnicMoreSophHopBnds10}) $\doteq$ can be replaced with a trivial inequality $\leq$ for our needs here.

Now, following what was done in Section \ref{sec:xiulow}, to make the bound in (\ref{eq:chpos9lriclift}) operational, the only thing left to consider is $I_{lric}(c_3^{(s)},\beta)$. One then trivially has
\begin{equation}
I_{lric}(c_3^{(s)},\beta)  =  \frac{1}{nc_3^{(s)}}\log(E(\max_{\x\in\Sric}(e^{-c_3^{(s)}\sqrt{n}\h^T\x}))) =  \frac{1}{nc_3^{(s)}}\log(E(\max_{\x\in\Sric}(e^{c_3^{(s)}\sqrt{n}\h^T\x}))).\label{eq:defIsl1}
\end{equation}
Comparing (\ref{eq:defIsl1}) and (\ref{eq:defIs}) one then has
\begin{equation}
I_{lric}(c_3^{(s)},\beta) = I_{uric}(c_3^{(s)},\beta).\label{eq:defIsl2}
\end{equation}
Moreover, one can then use (\ref{eq:defbigIuric})-(\ref{eq:defbigIuric1}) to characterize $I_{lric}(c_3^{(s)},\beta)$
which is then sufficient to compute lower bounds on $E\xi_{lric}$.

We summarize the above results related to the lower bound of $E\xi_{lric}$ in the following theorem.

\begin{theorem}($E\xi_{lric}$ - lifted lower bound)
Let $A$ be an $m\times n$ matrix with i.i.d. standard normal components. Let $k,m,n$ be large
and let $\alpha=\frac{m}{n}$ and $\beta=\frac{k}{n}$ be constants
independent of $m$ and $n$. Further, let $\Sric$ be as defined in (\ref{eq:defSric}) (or in (\ref{eq:defSriclift})). Let $\mbox{erf}$ be the standard error function associated with zero-mean unit variance Gaussian random variable and let $\mbox{erfc}=1-\mbox{erf}$.
Let
\begin{equation}
\widehat{\gamma_{sph}^{(s)}}=\frac{2c_3^{(s)}-\sqrt{4(c_3^{(s)})^2+16\alpha}}{8},\label{eq:gamasphthmlriclift}
\end{equation}
and
\begin{equation}
I_{sph}(c_3^{(s)},\alpha)=
\left ( \widehat{\gamma_{sph}^{(s)}}-\frac{\alpha}{2c_3^{(s)}}\log(1-\frac{c_3^{(s)}}{2\widehat{\gamma_{sph}^{(s)}}}\right ).\label{eq:Isphthmlriclift}
\end{equation}
Further, let $c_3^{(s)}$ and $\gamma_{lric}^{(s)}$ be such that $\frac{c_3^{(s)}}{\gamma_{lric}^{(s)}}<\frac{1}{2}$. Also,
let $I^{(uric)}$ be defined through (\ref{eq:defbigIuric})-(\ref{eq:defbigIuric1}) and let
\begin{equation}
I_{lric}(c_3^{(s)},\beta)=I_{uric}(c_3^{(s)},\beta)=\min_{\gamma_{uric}^{(s)}\geq c_3^{(s)}/2,\nu_{uric}^{(s)}\geq 0}(\nu_{uric}^{(s)}\beta+ \gamma_{uric}^{(s)}+\frac{1}{c_3^{(s)}}\log(I^{(uric)})).\label{eq:Ilricthmlric}
\end{equation}
Then
\begin{equation}
\lim_{n\rightarrow\infty}\frac{E\xi_{lric}}{\sqrt{m}}=\lim_{n\rightarrow\infty}\frac{E(\min_{\x\in\Sric}\|A\x\|_2)}{\sqrt{m}}
\geq
\frac{1}{\sqrt{\alpha}}\min_{c_3^{(s)}\geq 0}-\left (-\frac{c_3^{(s)}}{2}+I_{lric}(c_3^{(s)},\beta)+I_{sph}(c_3^{(s)},\alpha)\right ).\label{eq:ubthmlriclift}
\end{equation}
Moreover, let $\xi_{lric}^{(l,lift)}$ be a quantity such that
\begin{equation}
\frac{1}{\sqrt{\alpha}}\min_{c_3^{(s)}\geq 0}\left (-\frac{c_3^{(s)}}{2}+I_{lric}(c_3^{(s)},\beta)+I_{sph}(c_3^{(s)},\alpha)\right )
>\xi_{lric}^{(l,lift)}.
\label{eq:condlricthm}
\end{equation}
Then
\begin{eqnarray}
& & \lim_{n\rightarrow\infty}P(\min_{\x\in\Sric}(\|A\x\|_2)\geq \xi_{lric}^{(l,lift)}\sqrt{m})\geq 1\nonumber \\
& \Leftrightarrow & \lim_{n\rightarrow\infty}P(\xi_{lric}\geq \xi_{lric}^{(l,lift)}\sqrt{m})\geq 1 \nonumber \\
& \Leftrightarrow & \lim_{n\rightarrow\infty}P(\xi_{lric}^2\geq (\xi_{lric}^{(l,lift)})^2 m)\geq 1. \label{eq:lricprobthm}
\end{eqnarray}
\label{thm:thmlriclift}
\end{theorem}
\begin{proof}
The first part follows from the above discussion. The moreover part follows from considerations presented in \cite{StojnicCSetam09,StojnicHopBnds10,StojnicMoreSophHopBnds10}.
\end{proof}

We will below present the results one can get using the above theorem. However, as we did in Section \ref{sec:xiulow}, before proceeding with the discussion of the results one can obtain through  Theorem \ref{thm:thmlriclift}, we also mention that the results presented in the previous section (essentially in Lemma \ref{lemma:lriclemma}) can in fact be deduced from the above theorem. Namely, in the limit $c_3^{(s)}\rightarrow 0$, one from (\ref{eq:chpos8lriclift}) has that $E\min_{\x\in\Sric}\h^T\x+\sqrt{\alpha n}$ can be used as an upper bound on $E\xi_{lric}$. This is of course exactly the same expression that was considered in the previous section. For the completeness we present the following corollary where we actually derive the results from the previous section as a special case of those given in the above theorem (of course, the special case actually assumes  $c_3^{(s)}\rightarrow 0$).

\begin{corollary}($E\xi_{lric}$ - upper bound)
Assume the setup of Theorem \ref{thm:thmlriclift}. Let $c_3^{(s)}\rightarrow 0$. Then
\begin{equation}
\widehat{\gamma_{sph}^{(s)}}\rightarrow -\frac{\sqrt{\alpha}}{2},\label{eq:gamasphcorlric}
\end{equation}
and
\begin{equation}
I_{sph}(c_3^{(s)},\alpha)\rightarrow
-\sqrt{\alpha}.\label{eq:Isphcorlric}
\end{equation}
Moreover, as in (\ref{eq:Iuriccoruric})
\begin{eqnarray}
\hspace{-.4in}I_{uric}(c_3^{(s)},\beta) & \rightarrow & \min_{\gamma_{uric}^{(s)},\nu\geq 0}
\left (\frac{\nu^2\beta}{4\gamma_{uric}^{(s)}}+\gamma_{sec}^{(s)}+\frac{(1-\nu^2)\mbox{erfc}(\nu/\sqrt{2})+
\frac{2}{\sqrt{\pi}}\frac{\nu}{\sqrt{2}}e^{-\frac{\nu^2}{2}}}{4\gamma_{sec}^{(s)}}\right )\nonumber \\
& = & \min_{\nu\geq 0}
\sqrt{\left (\beta\nu^2+
\mbox{erfc}(\frac{\nu}{\sqrt{2}})(1-\nu^2)+\frac{2\nu e^{-\frac{\nu^2}{2}}}{\sqrt{2\pi}})\right )}.\label{eq:Ilriccorlric}
\end{eqnarray}
Choosing $\nu=\sqrt{2}\mbox{erfinv}(1-\beta)$
one then has
\begin{equation}
\lim_{n\rightarrow \infty}\frac{E\xi_{lric}}{\sqrt{m}}=\frac{E(\min_{\x\in\Sric} \|A\x\|_2)}{\sqrt{m}} \geq 1-\frac{1}{\sqrt{\alpha}}\sqrt{\beta+\frac{2\mbox{erfinv}(1-\beta)}{\sqrt{\pi}e^{(\mbox{erfinv}(1-\beta))^2}}}.\label{eq:lricexpcor}
\end{equation}
\label{cor:corlric}
\end{corollary}
\begin{proof}
Theorem \ref{thm:thmuriclift} holds for any $c_3^{(s)}\geq 0$. The above corollary instead of looking for the best possible $c_3^{(s)}$ in Theorem \ref{thm:thmuriclift} assumes a simple $c_3^{(s)}\rightarrow 0$ scenario. The rest of the proof follows the proof of Corollary \ref{cor:coruric}.

Alternatively, as mentioned above, one can look at $\frac{E\min_{\x\in\Sric}\h^T\x}{\sqrt{m}}+1$ and following the methodology presented in (\ref{eq:uriceq01}) (and originally in \cite{StojnicCSetam09}) obtain for a scalar $\nu=\sqrt{2}\mbox{erfinv}(1-\beta)$
\begin{equation}
\frac{E\min_{\x\in\Sric}\h^T\x}{\sqrt{m}} +1 \leq -\frac{1}{\sqrt{\alpha}}\sqrt{E_{\nu\leq |\h_i|}|\h_i|^2} +1.\label{eq:corlric0}
\end{equation}
Solving the integral (and using all the concentrating machinery of \cite{StojnicCSetam09}) one can write
\begin{equation}
\frac{E\min_{\x\in\Sric}\h^T\x}{\sqrt{m}}+1  \doteq - \frac{1}{\sqrt{\alpha}}\sqrt{\left (\int_{\nu\leq |\h_i|}|\h_i|^2\frac{e^{-\frac{\h_i^2}{2}}d\h_i}{\sqrt{2\pi}}\right )}+1= -\frac{1}{\sqrt{\alpha}} \sqrt{\left (\mbox{erfc}(\frac{\nu}{\sqrt{2}})+\frac{2\nu e^{-\frac{\nu^2}{2}}}{\sqrt{2\pi}}\right )}+1.\label{eq:corlric1}
\end{equation}
Connecting beginning and end in (\ref{eq:corlric1}) then leads to the condition given in the above corollary.
\end{proof}


\subsection{Numerical results -- lifted lower bound on $\xi_{lric}$}
\label{sec:xilprobnumlift}

In this subsection we present a small collection of numerical results one can obtain based on Theorem \ref{thm:thmlriclift}. In Table \ref{tab:lriclifttab1det} we show the upper bounds on $\lim_{n\rightarrow\infty}\frac{E\xi_{lric}}{\sqrt{m}}$ one can obtain based on Theorem \ref{thm:thmlriclift}. We refer to those bounds as $\xi_{lric}^{(l,lift)}$. Also, to get a feeling how the results of Theorem \ref{thm:thmuriclift} fare when compared to the ones presented in the previous section we in Table \ref{tab:lriclifttab1} also present the results we obtained in Subsection \ref{sec:xilprobnum} (which are of course based on Lemma \ref{lemma:lriclemma}  and Corollary \ref{cor:corlric}). For completeness, we in Table \ref{tab:lriclifttab1} also recall on the results from \cite{BT10}. Moreover, we show only what we call low $\beta/\alpha$ regime (i.e. $\beta/\alpha\leq 0.5$ regime). As $\beta/\alpha$ grows the values of bounds become smaller and their usefulness (as well as usefulness of original $E\xi_{lric}$ quantities) may not be of interest in such a regime.

\begin{table}
\caption{Lifted lower bounds on $\lim_{n\rightarrow\infty}\frac{E\xi_{lric}}{\sqrt{m}}$ -- low $\beta/\alpha\leq 0.5$ regime; optimized parameters}\vspace{.1in}
\hspace{-0in}\centering
\begin{tabular}{||l|c|c|c|c|c||}\hline\hline
 \hspace{.5in}$\alpha$                                                        & $0.1$ & $0.3$ & $0.5$ & $0.7$ & $0.9$  \\ \hline\hline
 $\beta/\alpha=0.05$; $c_{3}^{(s)}$                                           & $0.4592$ & $0.6653$ & $0.7756$ & $0.8494$ & $0.9027$ \\ \hline
 $\beta/\alpha=0.05$; $\nu_{lric}^{(s)}$                                      & $13.265$ & $7.1134$ & $5.2568$ & $4.2784$ & $3.6512$ \\ \hline
 $\beta/\alpha=0.05$; $\gamma_{lric}^{(s)}$                                   & $0.2399$ & $0.3546$ & $0.4195$ & $0.4654$ & $0.5006$ \\ \hline
 $\beta/\alpha=0.05$; $\xi_{lric}^{(l,lift)}$                                 & $0.4446$ & $0.4826$ & $0.5025$ & $0.5166$ & $0.5278$ \\ \hline\hline

 $\beta/\alpha=0.1$; $c_{3}^{(s)}$                                            & $0.7827$ & $1.0607$ & $1.1883$ & $1.2593$ & $1.2982$ \\ \hline
 $\beta/\alpha=0.1$; $\nu_{lric}^{(s)}$                                       & $7.5090$ & $4.1520$ & $3.0940$ & $2.5209$ & $2.1448$ \\ \hline
 $\beta/\alpha=0.1$; $\gamma_{lric}^{(s)}$                                    & $0.4017$ & $0.5545$ & $0.6310$ & $0.6790$ & $0.7110$ \\ \hline
 $\beta/\alpha=0.1$; $\xi_{lric}^{(l,lift)}$                                  & $0.2882$ & $0.3355$ & $0.3618$ & $0.3811$ & $0.3969$ \\ \hline\hline

 $\beta/\alpha=0.3$; $c_{3}^{(s)}$                                            & $4.0283$ & $3.8434$ & $3.5527$ & $3.2153$ & $2.8334$ \\ \hline
 $\beta/\alpha=0.3$; $\nu_{lric}^{(s)}$                                       & $1.5926$ & $1.1784$ & $0.9925$ & $0.8633$ & $0.7545$ \\ \hline
 $\beta/\alpha=0.3$; $\gamma_{lric}^{(s)}$                                    & $2.0184$ & $1.9363$ & $1.8042$ & $1.6528$ & $1.4850$ \\ \hline
 $\beta/\alpha=0.3$; $\xi_{lric}^{(l,lift)}$                                  & $0.0510$ & $0.0865$ & $0.1130$ & $0.1368$ & $0.1599$ \\ \hline\hline

 $\beta/\alpha=0.5$; $c_{3}^{(s)}$                                            & $37.468$ & $18.912$ & $12.497$ & $8.6351$ & $5.7138$ \\ \hline
 $\beta/\alpha=0.5$; $\nu_{lric}^{(s)}$                                       & $0.2144$ & $0.2928$ & $0.3337$ & $0.3570$ & $0.3593$ \\ \hline
 $\beta/\alpha=0.5$; $\gamma_{lric}^{(s)}$                                    & $18.735$ & $9.4602$ & $6.2593$ & $4.3411$ & $2.9056$ \\ \hline
 $\beta/\alpha=0.5$; $\xi_{lric}^{(l,lift)}$                                  & $0.0041$ & $0.0136$ & $0.0252$ & $0.0397$ & $0.0590$ \\ \hline\hline
\end{tabular}
\label{tab:lriclifttab1det}
\end{table}

As can be seen from the table, not only are the results from Theorem \ref{thm:thmlriclift} conceptually substantially better than the counterparts given in Lemma \ref{lemma:lriclemma}, they are also capable of offering substantial practical improvement over counterparts from Lemma \ref{lemma:lriclemma} (in fact, they also improve on the results from \cite{BT10}). Of course one then wonders how far from the optimal are the results that we presented. Well, as usual, there are certain obvious limitations and those relate to the numerical nature of the provided results. Namely, we, as in Section \ref{sec:xiulow}, solved the numerical optimizations that appear in Theorem \ref{thm:thmlriclift} only on a local optimum level and obviously only with a finite precision. We do not know if a substantial change would occur in the presented results had we solved them on a global optimum level (we recall that finding local optima is of course certainly enough to establish valid lower bounds; moreover in Table \ref{tab:lriclifttab1det} we provide a detailed values for optimizing parameters that we chose). As for our original question related to how far away from the true $E\xi_{lric}$ the results presented in Table \ref{tab:lriclifttab1} are, we actually believe that a unique answer is a bit hard to provide (it is highly likely that such an assessment may depend on the values $\beta$ and $\alpha$ take).

\begin{table}
\caption{Lifted lower bounds on $\lim_{n\rightarrow\infty}\frac{E\xi_{lric}}{\sqrt{m}}$ -- low $\beta/\alpha\leq 0.5$ regime}\vspace{.1in}
\hspace{-0in}\centering
\begin{tabular}{||l|c|c|c|c|c||}\hline\hline
 \hspace{1in}$\alpha$                                                   & $0.1$ & $0.3$ & $0.5$ & $0.7$ & $0.9$  \\ \hline\hline
 $\beta/\alpha=0.05$; $\xi_{lric}^{BT}$                                 & $0.4224$ & $0.4545$ & $0.4709$ & $0.4823$ & $0.4911$ \\ \hline
 $\beta/\alpha=0.05$; $\xi_{lric}^{(l)}$  ($c_3^{(s)}\rightarrow 0$)    & $0.3031$ & $0.3789$ & $0.4168$ & $0.4429$ & $0.4631$ \\ \hline
 $\beta/\alpha=0.05$; $\xi_{lric}^{(l,lift)}$  (optimized $c_3^{(s)}$)  & $0.4446$ & $0.4826$ & $0.5025$ & $0.5166$ & $0.5278$ \\ \hline\hline

 $\beta/\alpha=0.1$; $\xi_{lric}^{BT}$                                  & $0.2717$ & $0.3120$ & $0.3335$ & $0.3489$ & $0.3611$ \\ \hline
 $\beta/\alpha=0.1$; $\xi_{lric}^{(l)}$   ($c_3^{(s)}\rightarrow 0$)    & $0.0808$ & $0.1951$ & $0.2529$ & $0.2929$ & $0.3239$ \\ \hline
 $\beta/\alpha=0.1$; $\xi_{lric}^{(l,lift)}$   (optimized $c_3^{(s)}$)  & $0.2882$ & $0.3355$ & $0.3618$ & $0.3811$ & $0.3969$ \\ \hline\hline

 $\beta/\alpha=0.3$; $\xi_{lric}^{BT}$                                  & $0.0488$ & $0.0803$ & $0.1025$ & $0.1215$ & $0.1389$ \\ \hline
 $\beta/\alpha=0.3$; $\xi_{lric}^{(l)}$   $c_3^{(s)}\rightarrow 0$)     & $-0.394$ & $-0.171$ & $-0.056$ & $0.0247$ & $0.0877$ \\ \hline
 $\beta/\alpha=0.3$; $\xi_{lric}^{(l,lift)}$   (optimized $c_3^{(s)}$)  & $0.0510$ & $0.0865$ & $0.1130$ & $0.1368$ & $0.1599$ \\ \hline\hline

 $\beta/\alpha=0.5$; $\xi_{lric}^{BT}$                                  & $0.0041$ & $0.0130$ & $0.0234$ & $0.0356$ & $0.0504$ \\ \hline
 $\beta/\alpha=0.5$; $\xi_{lric}^{(l)}$   $c_3^{(s)}\rightarrow 0$)     & $-0.670$ & $-0.363$ & $-0.203$ & $-0.090$ & $-0.002$ \\ \hline
 $\beta/\alpha=0.5$; $\xi_{lric}^{(l,lift)}$   (optimized $c_3^{(s)}$)  & $0.0041$ & $0.0136$ & $0.0252$ & $0.0397$ & $0.0590$ \\ \hline\hline
\end{tabular}
\label{tab:lriclifttab1}
\end{table}

\section{Conclusion}
\label{sec:conc}

In this paper we looked at random matrices and studied their a particular property called restricted isometry. We developed a couple of mechanisms that can be utilized to estimate the values of the so-called isometry constants (quantities one typically associates with the isometry property).

To be a bit more specific, we designed a mechanism based on our recent results from \cite{StojnicCSetam09} that provides a fairly good set of estimates for the upper isometry constants. However, when adapted to cover the lower isometry constants it did not achieve the same success. We then went further and attempted to utilize some of the ideas we developed in \cite{StojnicMoreSophHopBnds10,StojnicLiftStrSec13} to lower the upper and to lift the lower isometry constants. The proposed methodology worked fairly well and the improvements we obtained (especially when it comes to the lower isometry constants were substantial). Moreover, in a wide range of problem parameters (dimensions) we feel confident that the results we obtained are actually fairly close to the exact ones.

As was the case in \cite{StojnicCSetam09,StojnicMoreSophHopBnds10,StojnicLiftStrSec13}, the purely theoretical results we presented are for the so-called Gaussian models, i.e. for systems with i.i.d. Gaussian coefficients. Such an assumption significantly simplified our exposition. However, all results that we presented can easily be extended to the case of many other models of randomness. There are many ways how this can be done. Instead of recalling on them here we refer to a brief discussion about it that we presented in \cite{StojnicMoreSophHopBnds10}.

As for usefulness of the presented results, there is hardly any limit. First, one can look at a host of related problems from the compressed sensing literature. Pretty much any problem that is typically attacked through the isometry constants can now be revisited. On a more mathematical side, isometry constants are tightly connected with the condition numbers of random matrices and the estimates we provided here will be of help when studying many problems where variants of condition numbers appear.

Also, on a purely mathematical side, one can observe that the isometry properties that we considered in this paper are based on $\ell_2/\ell_2$ isometries. Of course, one can define a tone of other isometries and for pretty much any of them the methods proposed here work (in fact for some of them they actually work even better than for those considered here). We will present some of these applications in a few forthcoming papers.

\begin{singlespace}
\bibliographystyle{plain}
\bibliography{RicBndsRefs}
\end{singlespace}

\end{document}